\documentclass[11pt, reqno]{amsart}
\usepackage[title]{appendix}
\usepackage{multirow}

\usepackage{euscript,amscd,amsgen,amsfonts,amssymb}

\newcommand{\eqdef}{\stackrel{\scriptscriptstyle\rm def}{=}}
\usepackage{bbm}
\usepackage{pgfplots, multicol}
\usepackage{mathtools}
\usepackage{paralist}
\usepackage{todonotes}
\usepackage{enumerate}
\usepackage{tikz-cd}
\usepackage[pdf]{pstricks}
\usepackage{ebezier}

\pgfplotsset{compat=1.12}

\newtheorem{theorem}{Theorem}

\newtheorem{proposition}{Proposition}
\newtheorem{question}{Question}
\newtheorem{corollary}{Corollary}
\newtheorem{definition}{Definition}
\newtheorem{lemma}{Lemma}
\newtheorem{remark}{Remark}

\newtheorem{example}{Example}

\newcommand{\beha}{\begin{enumerate}}
\newcommand{\behe}{\end{enumerate}}
\renewcommand{\epsilon}{\varepsilon}

\newcommand{\R}{{\rm Rot}}
\newcommand{\Per}{{\rm Per}}
\newcommand{\EPer}{{\rm EPer}}
\newcommand{\Ec}{{{\mathcal E}_{cc}}}

\newcommand{\Or}{\mathcal{O}}

\newcommand{\cU}{\mathcal{U}}

\newcommand{\cM}{\mathcal{M}}

\newcommand{\cH}{\mathcal{H}}
\newcommand{\bR}{{\mathbb R}}

\newcommand{\bZ}{{\mathbb Z}}
\newcommand{\bN}{{\mathbb N}}

\newcommand{\cA}{{\mathcal A}}
\newcommand{\cB}{{\mathcal B}}
\newcommand{\cC}{{\mathcal C}}

\def\1{1\!\!1}

\def\and{\text{ and }}

        \def\conv{\text{{\rm conv}}}

           \def\htop{{\text h_{\rm{top}}}}

\usepackage{pst-fill}

                        \def\^{\tilde}

\def\Per{{\rm Per}}
\def\inn{{\rm int}}
\def\ri{{\rm ri\,}}

\def\var{{\rm var}}

\def\Per{{\rm Per}}
\def\EPer{{\rm EPer}}

\def\1{1\!\!1}

\def\rv{{\rm rv}}
\def\inn{{\rm int\ }}

\newtheorem*{thmA}{Theorem A}

\newtheorem*{thmB}{Theorem B}

\DeclareMathSymbol{\varnothing}{\mathord}{AMSb}{"3F}
\renewcommand{\emptyset}{\varnothing}
\title[]{A topological classification of locally constant potentials  via zero-temperature measures}

\author{Christian Wolf}\address{Department of Mathematics, The City College of New York, New York, NY, 10031, USA\\
 The Graduate Center, CUNY,
                 New York, NY 10016 USA}\email{cwolf@ccny.cuny.edu}

\author{Yun Yang}\address{Mathematics Department
                 The Graduate Center, CUNY,
                 New York, NY 10016 USA }\email{yyang@gc.cuny.edu}

\thanks{We thank the anonymous referee for many helpful suggestions, which greatly improved the paper and generalized the results. We also thank Xufeng Guo for providing helpful computer simulations.}

\begin{document}
\maketitle

\begin{abstract}
We provide a topological classification of  locally constant functions over subshifts of finite type via their zero-temperature measures.  Our approach is to analyze the relationship between the distribution of the zero-temperature measures and the boundary of higher dimensional  generalized rotation sets.  We also discuss  the regularity of the  localized entropy function on the boundary of the generalized rotation sets. 
\end{abstract}
\maketitle
\section{Introduction}
\subsection{Motivation}
Recently, the following optimization problems have been posed for various systems.
 Consider a continuous map $f:X\rightarrow X$ on a compact metric space $X$ and a continuous real-valued function $\phi$ on $X$.
 \begin{question}[Equilibrium states]Which $f$-invariant measure(s) $\mu$ maximizes the topological pressure (called equilibrium state), i.e.
$$h_\mu(f)+\int \phi\, d\mu\geq h_\nu(f)+\int \phi\, d\nu$$
 for all $f$-invariant probability measures $\nu$? Here  $h_\nu(f)$ is the measure-theoretic entropy of $\nu$. We refer to \cite{CP} and the references therein for an overview and recent results on equilibrium states.
\end{question}
\begin{question}[Ergodic optimization] Which $f$-invariant measure(s) $\mu$ optimizes $\phi$ (called maximizing measure), i.e.
$$\int \phi\, d\mu\geq\int \phi\, d\nu$$
for all $f$-invariant ergodic probability measures $\nu$? We refer to \cite{Je2} for an introduction of the subject.
\end{question}
\begin{question}[Optimal orbits] The orbit of which point(s) $x_0$ yields the largest time average of $f$ (called optimal orbit), i.e.
$$\lim_{n\rightarrow \infty}\frac{1}{n}\sum_{i=0}^{n-1}\phi(f^i(x_0))\geq \lim_{n\rightarrow \infty}\frac{1}{n}\sum_{i=0}^{n-1}\phi(f^i(x))$$
for all $x\in X$, see e.g. \cite{Je2, YH}  for details.
\end{question}

The research towards answering these three questions is motivated by various problems in several areas. For example, the study of equilibrium states has its roots in statistical mechanics.  As for ergodic optimization and the classification of optimal orbits, one motivation goes back to the theory of Mather [17] and Ma\~n\'e [15, 16] on the dynamics of the Euler-Lagrange flow. 
At first glance, there is no connection between the three questions. However, Birkhoff's ergodic theorem states that
$$
\text{space average }= \int \phi\, d\mu= \text{ time average }=\lim_{n\rightarrow\infty}\frac{1}{n}\sum_{i=0}^{n-1}\phi(f^i(x)) $$
for any ergodic $f$-invariant measure $\mu$  and $\mu$-almost every point $x$. Hence, selecting optimal orbits is naturally related to finding  maximizing ergodic measures.  The following inequality builds a connection between equilibrium states and maximizing measures:
$$h_{\mu_T}(f)+\int \frac{1}{T}\phi\, d\mu_T\geq h_\nu(f)+\int \frac{1}{T}\phi\, d\nu,$$
i.e. $$Th_{\mu_T}(f)+\int \phi\, d\mu_T\geq Th_\nu(f)+\int \phi\, d\nu,$$
where $\mu_{T}$ is the equilibrium state for the potential $\frac{1}{T}\phi$. 
It follows that the limit $\lim_{T\rightarrow0}\mu_T$ (if it exists) is a particular  maximizing measure of the function $\phi$. 
 The limit $$\lim_{T\rightarrow0}\mu_T=\mu$$ is called 
the zero-temperature measure of $\phi$. Here the limit is taken in  the weak* topology.
If the limit does not exist, one considers the accumulation points of the sequence $(\mu_T)_T$   which are also maximizing measures. These measures are ground states and are of special interest in  statistical physics. We note that  here $T$ is interpreted as the temperature of the system.\footnote{We point out that in the mathematical theory of the thermodynamic formalism
it is customary to consider the inverse temperature $t=1/T$ and then to take the limit $t\to\infty$, see Section 2.5 for details. We mention that the notation that is being used for the inverse temperature in Physics is $\beta=\frac{1}{k_{B}T}$, where $k_{B}$ is Boltzmann's constant which can be taken equal to one in an approximate system of units.}

%

In this paper, we  study the zero-temperature measures for the locally constant functions over subshifts of finite type from the topological point of view.
Let $f:X\to X$  be a  subshift of finite type
and $\phi: X\rightarrow \mathbb{R}$ be a locally constant function.  
Such a function admits a unique equilibrium state $\mu_{\phi}$. We are interested in the zero temperature limit $ \mu_{\infty,\phi}$  given by 
$$\mu_{t\phi}\rightarrow \mu_{\infty,\phi}, \text{ as }t\rightarrow \infty.$$
The first natural question is whether the limit exists and if so, whether it can be described precisely or not. 
Nekhoroshev \cite{N} proved that generically, the zero-temperature limits of equilibrium states of locally constant potentials exist and are supported on  a periodic configuration.  Using analytic geometry, Br\'{e}mont \cite{Br}  proved  the convergence of $\mu_{t\phi}$, as $t\rightarrow \infty.$ Later, Leplaideur \cite{L} obtained a dynamical proof of the convergence and the limit was partially identified: If $\psi$ is H\"{o}lder continuous and $\phi$ is locally constant, then the limit of $\mu_{\psi+t \phi}$ exists, as $t\rightarrow \infty.$
Based on  different tools, namely the approximation by periodic orbits, the contraction-mapping approach to the Perron-Frobenius theorem for matrices, and the idea of a renormalization procedure, Chazottes, Gambaudo and Ugalde \cite{CGU} proved that as $t\rightarrow \infty$, the family $(\mu_{t\phi})_{t>0}$ converges  to a measure that is concentrated on a certain subshift of finite type, which is the union of transitive subshifts of finite type.   Recently, Contreras \cite{C} proved that for expanding transformations, the maximizing measures of a generic Lipschitz function are supported on a single periodic orbit. 

Other than results for convergence, there are also nonconvegence examples in the literature. 
Namely, Chazottes and Hochman \cite{CH} considered examples of the lattice $\bZ^d$ with finite alphabet for which the zero temperature limit does not exist.  In particular, for $d=1$ they derived non-existence for a Lipschitz potential as well as for $d\geq 3$ for a locally constant potential.  More recently, Coronel and Rivera-Letelier \cite{CR} provided  non-convergence examples in any dimension for Lipschitz potentials for both, finite alphabets and the unit circle as alphabet.  A non-converging example with a discontinuous function on the unit-circle alphabet over $\mathbb{Z}^d, d\geq 1$ was provided earlier by Van Enter and Ruszel \cite{VR}.


The goal in this paper is to classify the topological structure of the space of  locally constant functions over subshifts of finite type in terms of their zero-temperature measures.  Our approach is to apply the connection between the distribution of the zero-temperature measures and  the geometric properties of the generalized rotational sets (see \cite{KW4}).  In particular, we consider a higher dimensional function $\Phi$ that encodes all one-dimensional functions. We then obtain a topological classification of  the zero-temperature limits of locally constant functions in terms of geometric properties of the generalized rotation set of $\Phi$.

\subsection{Statement of the Results}
Let $f:X\to X$ be a subshift of finite type and let $LC_k(X,\mathbb{R})$ denote the space of functions that
are constant on cylinders of length $k$. Let $m_c(k)$ denote the cardinality of the set of cylinders of length $k$.
Thus, we can identify $LC_k(X,\mathbb{R})$ with $\mathbb{R}^{m_c(k)}$ which makes $LC_k(X,\mathbb{R})$ a Banach space when endowed with the standard norm. 
Let  $x\in X$ be a periodic point with period $n$. We denote by $\mu_x$ the unique invariant measure supported on the orbit of $x$ (see \eqref{perm} for details).
Moreover, we say  $x$  is $k$-elementary if the cylinders of length $k$ generated by $x,f(x),\dots,f^{n-1}(x)$ are pairwise disjoint (see Definition \ref{defkelementary}). The notion of elementary periodic orbits was first considered by Ziemian \cite{Z}.
 Our main result is a specific description of the zero-temperature measures for the functions in $LC_k(X,\mathbb{R})$. In particular, we construct a partition of $LC_k(X,\mathbb{R})$ into finitely many convex  cone components such that the zero-temperature measures of all functions in a component are  convex sums of finitely many ergodic measures. Moreover, the sets associated
with unique $k$-elementary periodic point zero-temperature measures form an open and dense subset of $LC_k(X,\mathbb{R})$. We are now able to formulate the main result of this paper:
\begin{thmA}\label{main theorem}
  Let $f:X\to X$ be a transitive subshift of finite type and let $k \in \bN$. Then there exist  a partition of $LC_k(X,\bR)$  into  finitely many convex cones $\cU_1,\ldots, \cU_{\ell_1}, \cU_{\ell_1+1},\ldots,\cU_{\ell_2},\cU_{\ell_2+1},\ldots, \cU_N$ and  finite sets of ergodic measures $\cM_1,\ldots, \cM_{\ell_1}, \cM_{\ell_1+1},\ldots,\cM_{\ell_2},\cM_{\ell_2+1},\ldots, \cM_N$ such that the following properties hold:
\begin{enumerate}
\item[{\rm (a)}] The sets $\cU_{1}, \ldots, \cU_{\ell_1}$ are open in $LC_k(X,\bR)$ and the closure of their union is $LC_k(X,\bR)$. For $1\leq i\leq \ell_1$, $\cM_i=\{\mu_{x_i}\}$, where $x_i$ is a $k$-elementary periodic point. Moreover, for any $\phi\in \cU_i$,  $\mu_{\infty,\phi}=\mu_{x_i}$,   and $\mu_{x_i}$ is the 
unique maximizing measure of $\phi$.
\item[{\rm (b)}] 
For $\ell_1<i\leq \ell_2$, $\cM_i=\{\mu_{i}\}$, where $\mu_i$ is the unique measure of maximal entropy of a non-discrete transitive subshift of finite type $X_i\subset X$. Moreover, for any $\phi\in \cU_i$, $\mu_{\infty,\phi}=\mu_i$. 

\item[{\rm (c)}]
For  $ \ell_2< i\leq N$,  $\cM_i=\{\mu^1_i, \ldots, \mu^{n_i}_i\}$, where $n_i\geq 2$ and $\mu^j_{i}$ is the unique measure of maximal entropy of some transitive subshift of finite type $X^j_i$ with the same entropy, i.e. $h_{\mu^{j_1}_i}(f)= h_{\mu^{j_2}_i}(f)$ for $1\leq j_1,j_2\leq n_i$. 
Further, for any $\phi\in \cU_i,$ 
\begin{equation}\label{messum} 
\mu_{{\infty,\phi}}=a_{1,\phi}\mu^1_i+\dots+a_{n_i,\phi}\mu^{n_i}_i,
\end{equation}
where  $0\leq a_{j,\phi}\leq 1$ and $\sum a_{j,\phi}=1$. Moreover, in each $\cU_i$ there are only finitely many choices for the coefficients $a_{j,\phi}$.
\end{enumerate}
\end{thmA}
 

We remark that in part (b)  $\cU_{l_1+1}=\{\phi \text{ cohomologous to zero}\}$. We refer to Section 2.4 for details about cohomologous functions. In particular,  $X_{l_1+1}=X$ and $\mu_{\ell_1+1}=\mu_{\rm mme}$ is the unique measure of maximal entropy of $f$. 
Further, we point out that part (c)  covers both, the case of discrete and non-discrete subshifts of finite type.  Namely, in case of zero entropy the $X^{j}_i$'s are $k$-elementary periodic orbits and the positive entropy case corresponds to non-discrete subshifts of finite type $X^{j}_i$. We notice that in part (c) the coefficients $a_{j,\phi}$  are  in general not restricted to values in $(0,1)$. This is shown in Appendix A.1,  Example \ref{example2c}  where  we consider the case when $\cM_i$  contains 3 fixed point measures and depending on the function $\phi\in\cU_i$  either
none or one of the coefficients $a_{j,\phi}$ is zero.

Since the set of locally constant functions
$LC(X,\bR)=\bigcup_{k\in\mathbb{N}}LC_k(X,\bR),$
we have the following direct corollary.

\begin{corollary}
There exist countable sets $\mathcal{M}_{\rm univ}\subset \mathcal{M}_E$,  where $\mathcal{M}_E$ is the collection of ergodic invariant measures for $f$, and $E_{\rm univ}=\{h_{\mu}(f):\mu\in \cM_{\rm univ}\}\subset [0,h_{\rm top}(f)]$ such for all $\phi\in LC(X,\bR)$  the zero-temperature measure $ \mu_{\infty,\phi}$ is a finite convex sum of measures in $\mathcal{M}_{\rm univ}$ and $h_{ \mu_{\infty,\phi}}(f)\in E_{\rm univ}$.
\end{corollary}

We point out that Nekhoroshev \cite{N} considered the case of two-sided full-shifts and proved a non-explicit version of  Theorem A (a) and part of (c). In contrast, our result  is direct in the sense that we construct the convex cone components $\cU_i$ explicitly  through the direction vectors of the faces of a polyhedron rotation set associated with a  computable universal higher dimensional function (see Theorem \ref{finite measures}).
Combining this identification with the recent computability results for rotation sets  \cite{BSW}, we are able to explicitly compute the sets  $\cU_i$ and $\cM_i$. Another advantage of Theorem A is that its proof provides additional information about the relative position among the components $\cU_i$. For example, for any pair $\cU_{i_1}$ and $\cU_{i_2}$, it is possible to decide if $\cU_{i_1}$ and $\cU_{i_2}$ are connected by a path of functions that stays entirely in $\cU_{i_1}\cup \cU_{i_2}$. We also mention that our methods are geometric in nature and in particular do not make use of Ruelle-Perron-Frobenius's transfer operator theory. This  suggests that our methods could be extended to handle more general classes of systems and functions.


 As mentioned earlier, to obtain Theorem A we apply the theory of higher dimensional generalized rotation sets.
 For $\Phi\in C(X,\bR^m)$ we denote by $\R(\Phi)$ the \emph{generalized rotation set} of $\Phi$ which is the set of all $\mu$-integrals of $\Phi$, where $\mu$ runs over the  $f$-invariant probability measures  $\cM$ on $X$. We refer to Section 2.3 for references and details.
We develop versions of Theorem A for higher dimensional functions (Theorems \ref{dichotomy} and \ref{finite measures}). We then identify the zero-temperature measures $ \mu_{\infty,\phi}$ in Theorem A as certain entropy maximizing measures at the boundary of a universal rotation set. Naturally, this leads to the question how the entropy varies on the boundary of rotation sets.     

For $w\in \R(\Phi)$ we denote by 
 $\mathcal{M}_\Phi(w)=\{\mu\in \cM: \rv(\mu)=w\}$
 the \emph{rotation class} of $w$, where $\rv(\mu)=$ is the \emph{rotation vector} of $\mu$ (see \eqref{limzt}). 
Following \cite{Je, KW1}, we define the \emph{localized entropy} at $w\in \R(\Phi)$ by
\begin{equation}\label{defH}
\cH(w)=\cH_\Phi(w)\eqdef\sup\{h_\mu(f): \mu \in \cM_\Phi(w)\}.
\end{equation}
The upper semi-continuity of $\mu\mapsto h_\mu(f)$ implies that $w\mapsto \cH(w)$ is upper semi-continuous on $\R(\Phi)$. Further, since $\mu\mapsto h_\mu(f)$ is affine, $w\mapsto \cH(w)$ is concave and consequently continuous on the (relative) interior of $\R(\Phi)$.
Here the relative interior of a set is defined as the interior of the set considered as a subset of its affine hull. We note that, in general, $w\mapsto \cH(w)$ is not continuous on $\R(\Phi)$, see \cite{W}. However, if $\Phi\in LC(X,\bR^m)$ then 
by Ziemian's Theorem (see Theorem \ref{thmelemtentary} in the text),
$\R(\Phi)$ is a polyhedron and thus the  celebrated Gale-Klee-Rockafellar Theorem \cite{GKR} implies that $w\mapsto \cH(w)$ is continuous on $\R(\Phi)$.
It is shown in \cite{KW4} that for  H\"older continuous functions $\Phi$ (and therefore in particular for locally constant functions $\Phi\in LC(X,\bR^m)$) the localized entropy $w\mapsto \cH(w)$ is analytic on the interior of $\R(\Phi)$. 

Since the rotation set of $\Phi\in LC(X,\bR^m)$ is a polyhedron it follows that the faces of $\R(\Phi)$ are (lower-dimensional) polyhedra as well. One might expect that $\cH$ restricted to the \emph{relative interior} (which we abbreviate by $\ri$) of a face of $\R(\Phi)$ is also analytic. However, we are able to prove the following somewhat unexpected result (see  Proposition \ref{concenv}, Theorem \ref{thmpwanalytic} in the text, and Examples \ref{example3} and \ref{example44} in Appendix B).

\begin{thmB}\label{differentiability of entropy}
  Let  $\Phi\in LC(X, \bR^m)$ and let $F$ be a non-singleton face of $\R(\Phi)$. Then $\cH|_F$ is the concave envelop of finitely many concave functions $h_i:F_i\to \bR$  defined on sub-polyhedra $F_i\subset F$. Moreover, the functions $h_i$ are analytic on ${\rm ri}\, F_i$, but $\cH$ is in general not analytic on  $\ri F$. If $m=2$ then $\cH$ is a piecewise $C^1$-function on  $\partial \R(\Phi)$. \end{thmB}

This paper is organized as follows. In Section 2 we review some
basic concepts and results about symbolic dynamics, zero-temperature measures,  the thermodynamic formalism and generalized rotation sets.
 Section 3 is devoted to the study of zero-temperature measures of higher dimensional functions in terms of their rotation sets. In Section 4, we present the proof of  Theorem A.  Next, we study in Section 5 the regularity of the localized entropy function at the boundary of rotation sets. Finally, we provide in Appendices A and B some examples that illustrate applications of Theorems A and B, respectively.

\section{Preliminaries}\label{sec:2}

In this section we discuss relevant background material which will be used later on. We will continue to use the notations from Section 1. We start by recalling 
some basic facts from symbolic dynamics.

\subsection{Shift maps }

Let $d\in \bN$ and let $\cA=\{0,\dots,d-1\}$ be a finite alphabet in $d$ symbols. The (one-sided) shift space $\Sigma_d$ on the alphabet $\cA$ is the set of
all sequences $x=(x_n)_{n=1}^\infty$ where $x_n\in \cA$ for all $n\in \bN$.  
The shift map $f:\Sigma_d\to \Sigma_d$ (defined by $f(x)_n=x_{n+1}$) is a continuous $d$ to $1$ map on $\Sigma_d$.
In the following, we use the symbol $X$ for any shift space including the full shift $X=\Sigma_d$.
A particular class of shift maps are subshifts of finite type.
Namely, suppose $A$ is a $d\times d$ matrix with values in $\{0,1\}$, then consider the set of sequences given by  
$X=X_A=\{x\in \Sigma_d: A_{x_n,x_{n+1}}=1\}$.  The set $X_A$ is a closed (and, therefore, compact) $f$-invariant set, and we say that $f|_{X_A}$ a {\em subshift of finite type}. By reducing the alphabet, if necessary, we always assume that $\cA$ does not contain letters that do not occur in any of the sequences in $X_A$.

Let $f:X\to X$ be a subshift.
We say $t=t_1 \cdots t_k\in \cA^k$ is a block of length $k$  and write $|t|=k$.  Given another block $s=s_1\cdots s_l\in \cA^l$ we define the concatenation of  $t$ and $s$ by $ts=t_1\cdots t_ks_1\cdots s_l \in \cA^{k+l}   $. Further, $\epsilon$  denotes the empty  block.
 Given $x\in X$, we write $\pi_k(x)=x_1\cdots x_k\in \cA^k$.
We denote the  cylinder of length $k$ generated by $t$ by $\cC_k(t)=\{x\in X: x_1=t_1,\dots, x_k=t_k\}$.  Given $x\in X$ and $k\in \bN$, we call $\cC_k(x)=\cC_k(\pi_k(x))$ the  cylinder of length $k$ generated by $x$. 
Further, we call  $\Or(t)=t_1\cdots t_kt_1\cdots t_kt_1\cdots t_k\cdots \in X$
the  periodic point with period $k$ generated by $t$. 
We denote by  $\Per_n(f)$ the set of periodic points of $f$ with prime period $n$  and by $\Per(f)$ the set of periodic points of $f$. If $n=1$ we say $x$ is a fixed point of $f$.  In the following we always assume that $n$ is the prime period of $x$. Let $x\in \Per_n(f)$.  We call $\tau_x=\pi_k(x)=x_1\cdots x_{n}$ the generating segment of  $x$, that is $x=\Or(\tau_x)$.  
For $x\in \Per_n(f)$,  the unique invariant measure supported on the orbit of $x$ is given by 
\begin{equation}\label{perm}
\mu_x=\frac{1}{n}(\delta_x+\dots +\delta_{f^{n-1}(x)}), 
\end{equation} where
 $\delta_y$ denotes the  Dirac measure on $y$. We also call $\mu_x$ the periodic point measure of $x$. Obviously, $\mu_x=\mu_{f^l(x)}$ for all $l\in\bN$. We write $\cM_{\rm Per}=\{\mu_x: x\in \Per(f)\}$ and observe that $\cM_{\rm Per}\subset \cM_E$.

\begin{definition}\label{defkelementary}Let $x\in \Per_n(f)$. 
Fix $k\in \bN$. We say $x$ is a $k$-elementary periodic point with period $n$ if $\cC_k(f^i(x))\not=\cC_k(f^j(x))$ for all $i,j=0,\dots,n-1$ with $i\not=j$. In case $k=1$ we simply say $x$ is an elementary periodic point. We denote by $\EPer_n^k(f)$ the set of all $k$-elementary periodic points with period $n$ and by $\EPer^k(f)$ the set of all $k$-elementary periodic points. 
\end{definition}
We refer to Table 1 in Appendix A.2 for a list of examples of elementary periodic points.
\begin{remark} We note that the period of a $k$-elementary periodic point is at most $m_c(k)$. In particular,  $\EPer^k(f)$ is finite.
\end{remark}
 \begin{definition}\label{defkpermut}Let $x,y$ be $k$-elementary periodic points of $f$ with period $n$. We say  $x$ and $y$ are $k$-permutable  
if $$ \{\cC_k(f^i(x)): i=0,\dots ,n-1\}= \{\cC_k(f^i(y)): i=0,\dots ,n-1\}.$$
\end{definition} Clearly, being $k$-permutable   is an equivalence relation on the set of $k$-elementary periodic points.

We note that in this paper we consider the case of one-sided
shift maps. However, all our result carry over to the case of two-sided
shift maps. For details how to make the connection between one-sided shift
maps and two-sided shift maps we refer to \cite{Je}.

\subsection{Notations from convex geometry}
Next we recall  some notions from convex geometry (see e.g. \cite{R}). For $m$-dimensional vectors $u=(u_1,...,u_m)$ and $v=(v_1,...,v_m)$ we write $u\cdot v =u_1v_1+...+u_mv_m$ for the inner product of $u$ and $v$. Let $B(v,r)$ denote the open  ball about $v\in \bR^m$ of radius $r$ with respect to the Euclidean metric. A subset $K\in\bR^m$ is \emph{convex} if for all $u, v\in K$ we have that $tu+(1-t)v\in K$ for all $t\in(0,1)$. 
A point $w\in K$ is called an \emph{extreme point} of $K$ if
$w$ is not contained in any open line segment with endpoints in $K$, i.e. if for all $u, v\in K$ with $u\not=v$ we have $w\not=tu+(1-t)v$ for all $t\in (0,1)$. 
We say a convex set $K\subset \bR^m$ is a \emph{convex cone} if  $tx\in K$ for all $x\in K$ and $t>0$.

For $K\subset \bR^m$, 
 the \emph{convex hull} of $K$, denoted by $\conv(K)$,  is the smallest convex set containing $K$. 
We will work with the standard topology on $\bR^m$.
For $K\subset\bR^m$ we denote by $\inn K$ the \emph{interior} of $K$,  by $\partial K$ the \emph{boundary} of $K$ and by $\overline{K}$ the \emph{closure} of $K$.  The \emph{relative interior} and the \emph{relative boundary} of $K$, denoted by $\ri K$ and $\partial_{\rm rel} K$, are the interior and the boundary of $K$ with respect to the topology on the affine hull of $K$, respectively.

For a non-zero vector $\alpha\in\bR^m$ and $a\in \bR$ the hyperplane $H=H_{\alpha,a}\eqdef\{u\in \bR^m: u\cdot \alpha= a\}$ is said to \emph{cut} $K$ if both open half spaces determined by $H$ contain points of $K$. Here $\alpha$ is a normal vector to $H$. We say that $H$ is a \emph{supporting hyperplane} for $K$ if its distance to $K$ is zero but it does not cut $K$.

A set $F\subset K$ is a \emph{face} of $K$ if there exists a supporting hyperplane $H$ such that $F=K\cap H.$ We say a normal vector $\alpha$ to $H$ is \emph{pointing away} from $K$ if for $w\in H$ the point $\alpha+w$ belongs to the open half space of $\bR^m\setminus H$ that does not intersect  $K$. We note that if $K$ has a non-empty interior then there exists a unique unit normal vector to $H$ that is pointing away from $K$.
A point $w\in K$ is called \emph{exposed} if $\{w\}$ is a face of $K$. We say  $K$ is a polyhedron if it is the convex hull of finitely many points in $\bR^m$. In this case the vertices
of $K$ coincide with the exposed points of $K$.

\subsection{Generalized rotation sets}
Given an $m$-dimensional continuous real-valued function (also called an $m$-dimensional potential) $\Phi=(\phi_1,\dots, \phi_m)$, we define the \emph{generalized rotation set} of $\Phi$ under the dynamics $f$ by
\begin{equation}\label{defrot}
\R(\Phi)=\text{Rot}(\Phi, f)=\left\{\rv(\mu):\, \mu\in\mathcal{M}\right\},
\end{equation}
where 
\begin{equation}\label{limzt}
\text{rv}(\mu)=\left(\int\phi_1\, d\mu,\dots, \int\phi_m\, d\mu\right)
\end{equation}
is the rotation vector of $\mu$ and $\cM=\mathcal{M}_f$ denotes the set of $f$-invariant Borel probability measures on $X$ endowed with weak* topology. We recall that this topology makes $\cM$ a compact convex metrizable topological space. 
If $m=1$ we use the  notation $\mu(\phi)=\int \phi\, d\mu$ instead of $\rv(\mu)$.
It follows from the definition that
the rotation set is a compact and convex subset of $\bR^m$. 
 We note that in general the geometry of rotation sets can be quite complicated. Indeed, it is proved in \cite{KW1} that  every compact and convex set $K\subset\bR^m$ is attained as the rotation set of some $m$-dimensional function $\Phi$. 
We point out that generalized rotation sets can be considered as  generalizations
of Poincar\'e's rotation number of an orientation preserving homeomorphism, see  \cite{GM,Je,KW1,MZ,Z} for further references and details.

\subsection{Locally constant functions.}Let $f:X\to X$ be a subshift of finite type on the alphabet $\cA=\{0,\dots,d-1\}$ with transition matrix $A$. Let  $m\in \bN$ and $\Phi\in C(X,\bR^m)$.  Given $k\in\bN$ we define $$\var_k(\Phi)=\sup\{\|\Phi(x)-\Phi(y)\|: x_1=y_1,\dots, x_k=y_k\}.$$
We say $\Phi$ is constant on cylinders of length $k$ if $\var_k(\Phi)=0$. It is easy to see that $\Phi$ is locally constant if and only if $\Phi$ is constant on cylinders of length $k$ for some $k\in\bN$. We denote by $LC_k(X,\bR^m)$ the set of all $\Phi$ that are constant on cylinders of length $k$.

Recall the definition of the rotation set $\R(\Phi)$ of $\Phi$ in Equation \eqref{defrot}.
Based on work of  Ziemian \cite{Z} and Jenkinson \cite{Je}, we provide the necessary tools for the study of zero-temperature measures and rotation sets. We start with the following elementary result.

\begin{proposition}\label{propone}
 Let $k,m\in \bN$, $\Phi\in LC_k(X,\bR^m)$, and  $d'=m_c(k)$. Then there exist a subshift $g:Y\to Y$ of finite type with alphabet $\cA'=\{0,\dots,d'-1\}$ and 
transition matrix $A'$, and a homeomorphism $h:X\to Y$ that conjugates $f$ and $g$ (i.e. $h\circ f=g\circ h$) such that the transition matrix $A'$ has at most $d$ non-zero entries in each row, and the function $\Phi'=\Phi\circ h^{-1}$  is constant on cylinders of length one. Moreover, $\R(\Phi')=\R(\Phi)$ and $\cH_{\Phi'}=\cH_\Phi$.
\end{proposition}
\begin{proof}
The proof is elementary and can be  for example  found in \cite{BSW}. For our exposition we only require the definitions of the subshift $g$ and the conjugating map $h$.
 Let $\{\cC_k(0),\dots,\cC_k({m_c(k)-1})\}$ denote the set of cylinders of length $k$ in $X$, which we identify with $\cA'=\{0,\dots, d'-1\}$. The transition matrix $A'$  is defined by  $a'_{i,j}=1$ if and only if there exists $x\in X$ with $\cC_k(x)=i$ and $\cC_k(f(x))=j$. 
Let $Y=Y_{\cA'}$ be the shift space in ${\cA'}^\bN$ given by the transition matrix $A'$. Furthermore, let $g:Y\to Y$ be the corresponding map for the subshift of finite type. For $x\in X$ we define $h(x)=y=(y_n)_{n=1}^\infty$ by $y_n=\cC_k(f^{n-1}(x))$. 
\end{proof}

Ziemian \cite{Z} proved that the rotation set of a function $\Phi$ that is constant on cylinders of length two is a polyhedron. This result extends  to functions that are constant on cylinders of length $k\geq 1$, see e.g. \cite{BSW, Je}. 

\begin{theorem}\label{thmelemtentary}
Let $f:X\to X$ be a transitive subshift of finite type and let $\Phi\in LC_k(X,\bR^m)$. Then $\R(\Phi)$ is a polyhedron, in particular $\R(\Phi)$ is the convex hull
of $\rv(\{\mu_x: x\in \EPer_k(f)\})$.
\end{theorem}

Next, we discuss the set of measures whose rotation vectors belong to a face of $\R(\Phi)$.  We follow Jenkinson \cite{Je} for our exposition.
Let $f:X\to X$ be a transitive subshift of finite type with transition matrix $A$. Moreover, let
$\Phi\in LC_k(X,\bR^m)$. By Theorem \ref{thmelemtentary}, the rotation set $\R(\Phi)$ is a polyhedron. Let $F$ be a face of $\R(\Phi)$, and let $x^1,\dots,x^\ell$ denote the $k$-elementary
periodic points whose rotation vectors lie in $F$.  Note that, in this list of $k$-elementary periodic points, points in the same orbit are considered distinct.
It follows from Theorem \ref{thmelemtentary} that $F=\conv(\rv(\mu_{x^1}),\dots,\rv(\mu_{x^\ell}))$. For each $r\in\{1,\dots,\ell\}$, let $p(r)$ be the period of $x^r$. Further, let 
$s^r=x_1^r\cdots x^r_{p(r)}$ be the generating segment of the periodic point $x^r$.  We define
\begin{equation}\label{perface}
X_F=\overline{\{x\in \Per(f): \rv(\mu_x)\in F\}}.
\end{equation}
The set $X_F$ coincides with of the set of all infinite $g$-admissible concatenations of the generating segments $s^1,\dots,s^\ell$ (see \cite{Je}). Here we say that the concatenation of $s^{r_1}$ with $s^{r_2}$ is $g$-admissible if the generating sequences of the periodic points $h(x^{r_1})$ and $h(x^{r_2})$ can be concatenated in the shift $g:Y\to Y$, where $h:X\to Y$ and $g:Y\to Y$ are as in  Proposition \ref{propone}.
The set of generating segments $s^1,\dots,s^\ell$ is an alphabet $\cB$ and so $f|_{X_F}$  
is a subshift of finite type with transition matrix $B=B(s^1, \dots,s^\ell)$ defined by the concatenation rules of $g$.  Recall, that  $f$ 
is called {\em non-wandering}, if the $f$ orbit of every point $x$ is non-wandering. It follows from \eqref{perface}, that $f|_{X_F}$ is non-wandering. 
\

Let $\mathcal{G}_B$ be Markov graph associated with the transition matrix $B$ and let 
$\mathcal{G}_1,\dots,\mathcal{G}_t$ denote the transitive components of  $\mathcal{G}_B$.
Moreover, for each $\mathcal{G}_i$, let $B_i$ denote the associated transition matrix.  Define $X_i=X_{B_i}$. Since $f|_{X_F}$ is non-wandering, it follows from the construction that 
\begin{equation}\label{sumtrans}
X_F=X_{1}\cup\dots \cup X_{t}.
\end{equation}
Observe that
$f|_{X_i}$ is a transitive subshift of finite type without transitions between sets other than $X_i$ (see, e.g., Corollary. 5.1.3 in \cite{Kit}). We note  that this includes the possibility of $X_i$ being a single periodic orbit.

\begin{theorem}\label{thmfaceshift}
Let $f:X\to X$ be a transitive subshift of finite type, and  let $\Phi\in LC_k(X,\bR^m)$. Let $F$ be a face of $\R(\Phi)$, and let $\mu\in \cM$. Then $\rv(\mu)\in F$ if and only if
${\rm supp}\, \mu\subset X_F$.
\end{theorem}
\begin{proof}
For  $k=2$ the assertion is proven in \cite{Je}. The general case can be easily deduced from the case $k=2$, Proposition \ref{propone}, and using the fact 
that $LC_1(X,\bR^m)\subset LC_2(X,\bR^m)$.
\end{proof}

\subsection{Equilibrium states, ground states and zero-temperature measures.}
Let $f:X\to X$ be a transitive subshift of finite type.
Given a continuous  function $\phi:X\to \bR $, we denote the \emph{topological pressure} of $\phi$ (with respect to $f$) by $P_{\rm top}(\phi)$ and the \emph{topological entropy} of $f$ by $h_{\rm top}(f)$ (see~\cite{Wal:81} for the definition and further details). \footnote{We note that in the respective literature of the mathematical thermodynamic formalism  the function $\phi$ is often referred  to as a potential.}

The topological pressure satisfies the well-known variational principle, namely,

\begin{equation}\label{eqvarpri}
P_{\rm top}(\phi)=
\sup_{\mu\in \cM} \left(h_\mu(f)+\int_X \phi\,d\mu\right).
\end{equation}
A measure
$\mu\in \cM$ that attains the supremum in \eqref{eqvarpri} is
called an \emph{equilibrium state} (also called equilibrium measure) of the function $\phi$. We denote by $ES(\phi)$ the set of all equilibrium states of $\phi$. 
Since $\mu\mapsto h_\mu(f)$
is upper semi-continuous we have $ES(\phi)\not=\emptyset$, in particular $ES(\phi)$ contains at least one ergodic equilibrium state.  Moreover, if $\phi$ is H\"older continuous
(and in particular if $\phi$ is locally constant) then $\phi$ has a unique equilibrium state that we denote by $\mu_\phi$. We recall that two H\"older continuous functions $\phi,\psi:X\to \bR$ are said to be cohomologous  if there exists a continuous function $\eta: X\to \bR$ and $K\in \bR$ such that $\phi-\psi=\eta\circ f-\eta +K$.  By Livschitz Theorem (e.g. \cite[Proposition 4.5]{Bo}), $\phi$ and $\psi$ are cohomologous if and only if $\mu_\phi=\mu_\psi$ if and only if there exists $K\in \bR$ such that
\begin{equation}\label{eqLiv}
\frac{1}{n}\sum_{k=0}^{n-1} \left(\phi(f^k(x))-\psi(f^k(x))\right) =K
\end{equation}
 for all $x\in \Per_n(f)$ and all $n\in\bN$. 
In particular, for $\phi,\psi\in LC_k(X,\mathbb{R})$ to be cohomologous, it is sufficient that $\eqref{eqLiv}$ holds for all  $x\in\EPer^k(f)$. This  follows from the fact that every periodic orbit can be written as a concatenation of generating segments of $k$-elementary periodic points.

Next we give the definition for ground states.
    We say $\mu\in \cM$ is a \emph{ground state} of the
function $\phi$ if there exists a sequence $(t_n)_n\subset \bR$ with $t_n\to \infty$, and  corresponding  equilibrium states  $\mu_{t_n\phi}\in ES(t_n\phi)$ such that $\mu_{t_n\phi}\to \mu$ as $n\to\infty$. Here we think of $t$ as the inverse temperature $1/ T$ of the system. Thus, a ground state is an accumulation point of equilibrium states when the temperature approaches zero.
We denote by $GS(\phi)$ the set of all ground states of $\phi$.

In order to define zero-temperature measures we  require convergence of the measures $\mu_{t\phi}$ rather than only convergence of a subsequence.
Namely, suppose there exists $t_0\geq 0$ such that  for all $t\geq t_0$ the function $t\phi$ has a unique equilibrium state. We say  $ \mu_{\infty,\phi}\in \cM$ is the \emph{zero-temperature measure} of the function $\phi$ if $ \mu_{\infty,\phi}$ is the weak$^\ast$  limit of the measures $\mu_{t\phi}$ as $t\to\infty$.
For $\Phi\in C(X,\bR^m)$ and $\alpha=(\alpha_1,\dots,\alpha_m)\in \bR^m$ we write $\alpha\cdot\Phi=\alpha_1\Phi_1+\dots+ \alpha_m\Phi_m$. If the limit
\begin{equation}
\mu_{\infty,\alpha\cdot \Phi}=\lim_{t\rightarrow\infty}\mu_{t\alpha\cdot \Phi}
\end{equation}
exists, we call $\mu_{\infty,\alpha\cdot \Phi}$
the zero temperature measure in direction $\alpha$.


\section{Classification of higher dimensional functions}

The strategy to prove  Theorem A is to construct a suitable higher dimensional function $\Phi$ that encodes all one-dimensional functions. With this goal in mind we provide in this section a classification of higher dimensional functions  in terms of the ``shape" of their generalized rotation sets. 
This classification will make use of  a result in \cite{KW4} that connects the geometry of the rotation set with the rotation vectors of the corresponding ground states. While this result is originally stated for continuous maps on compact metric spaces, here we only consider the case of subshifts of finite type.

We need the following notation.
Let $f:X\to X$ be a subshift of finite type, and let  $\Phi\in C(X,\bR^m)$ be a fixed $m$-dimensional function.
Given a direction vector $\alpha\in S^{m-1}$ we denote by $H_\alpha(\Phi)$ the supporting hyperplane of $\R(\Phi)$ for which $\alpha$ is the normal vector that points away from $\R(\Phi)$. Since $\R(\Phi)$ is a compact convex set, it follows from standard arguments in convex geometry that $H_\alpha(\Phi)$ is well-defined.
We denote by $F_\alpha(\Phi)\eqdef\R(\Phi)\cap H_\alpha(\Phi)$  the face of $\R(\Phi)$ associated with the direction vector $\alpha$. For any face $F$   of $\R(\Phi)$ we denote by $\alpha(F)=\{\alpha\in S^{m-1}: F_\alpha(\Phi)=F\}$  the set of direction vectors that generate $F$. Clearly, $\alpha(F)\not=\emptyset$ and $\alpha(F_i)\cap \alpha(F_j)=\emptyset$ if $F_i\not= F_j$.

\begin{theorem}[\cite{KW4}] \label{KW}Let $f:X\to X$ be a transitive subshift of finite type, let $\Phi\in C(X, \mathbb{R}^m)$  and  let $\alpha\in S^{m-1}$. Then

\begin{enumerate}
\item[(a)] If $\mu$ is a ground state of $\alpha\cdot \Phi$ then $\rv(\mu)\in F_{\alpha}(\Phi)$, and 
\begin{equation}
h_{\mu}(f)=\sup\{h_{\nu}(f):\rv(\nu)\in F_{\alpha}(\Phi)\}.
\end{equation}

\item[(b)] The set $\{\rv(\mu):\mu\in \text{GS}(\alpha\cdot \Phi)\}$ is compact. Further, if $\Phi$ is H\"older continuous then $\{\rv(\mu): \mu\in \text{GS}(\alpha\cdot \Phi)\}$ is connected. 

\end{enumerate}


\end{theorem}
We recall that for a locally constant function $\Phi$,  for each $\alpha\in S^{m-1}$ there exists  a unique ground state $\mu_{\infty,\alpha\cdot \Phi}$ of $\alpha\cdot \Phi$ , i.e., the zero-temperature limit exists.
In particular, in the locally constant case the set $\{\rv(\mu):\mu\in \text{GS}(\alpha\cdot \Phi)\}$ in Theorem \ref{KW} (b) is a singleton.
We note that Theorem \ref{KW} includes  the case $\inn \R(\Phi)=\emptyset$ in which case $F_{\alpha}(\Phi)=\R(\Phi)$ for all $\alpha\in S^{m-1}$ such that $\alpha\cdot \Phi$ is cohomologous to zero. In this situation $\mu_{\infty,\alpha\cdot \Phi}=\mu_{\rm mme}$  the unique measure of maximal entropy of $f$.

We continue to use the notations from Sections 1 and 2.  Let $f:X\to X$ be a transitive one-sided  subshift of finite type (see Section 2 for details). Let $k,m\in \bN$ and let $\Phi\in LC_k(X,\bR^m)$. By  Ziemian's Theorem (Theorem \ref{thmelemtentary} in this paper),  $\R(\Phi)$ is  a polyhedron.
  Let $V_\Phi=\{w_1,\dots,w_s\}$ denote the  vertex set of $\R(\Phi)$. Clearly, $s$ depends on $\Phi$. 
Again by Ziemian's Theorem, $\R(\Phi)$ is the convex hull of the rotation vectors of $k$-elementary periodic point measures (see Section 2 for the definition and details).
It follows that each $w_j\in V_\Phi$  has at least one $k$-elementary periodic point measure in  its rotation class $\mathcal{M}_{\Phi}(w_j)$. We will make use of the following definition.

\begin{definition}\label{defUk}
 We denote by $\cU(k)=\cU(k,m)$ the set of functions $\Phi$ in $LC_k(X,\bR^m)$ with the following properties.
\begin{enumerate}
\item[(a)]
If  $w$ is a vertex of $\R(\Phi)$ and if $x,y\in \EPer^k(f)$  with $\rv(\mu_x)=\rv(\mu_y)=w$, then $x$ and $y$ are $k$-permutable.  
\item[(b)]
If $x\in \EPer^k(f)$ with $\rv(\mu_x)\in \partial_{\rm rel} \R(\Phi)$ then  $\rv(\mu_x)$ is a vertex of $\R(\Phi)$.
\end{enumerate}
\end{definition}
It follows readily from Definition \ref{defUk} that all $k$-elementary periodic points  in a vertex rotation class have the same period.

 \begin{proposition}\label{propkey}
 Let  $\Phi\in \cU(k)$ and let $w$ be a vertex  of $\R(\Phi)$. Let $X_F$ denote the subshift of finite type associated with the face $F=\{w\}$ (see Equation \eqref{perface}).
 Then there exist $n\in\bN$ and $x\in\EPer^k_n(f)$ such that $X_F=\{x,f(x),\dots,f^{n-1}(x)\}$. In particular, $\cM_\Phi(w)=\{\mu_x\}$.\end{proposition}

 \begin{proof}
 Let $\cC_1,\dots,\cC_n$ denote the cylinders of length $k$ associated with $w$ (see Definition \ref{defUk}). Rather than working with $f$, we consider the conjugate system $g:Y\to Y$ defined in Proposition \ref{propone}. For $x\in \Per(f)$, by slightly abusing the notation, we continue to refer to the corresponding periodic point of $g$ as $x$. Recall that for the system $g$ each symbol corresponds to a cylinder of length $k$ in $X$.  
We will make use of the following elementary 
 Observation: If $x$ is a periodic point in $X_F$ then  $x$ has period $np$ for some $p\in\bN$. Moreover, every symbol (i.e. $k$-cylinder) occurs in the generating sequence $\tau_x$
 of $x$ precisely $p$ times.
  We note that the Observation is a consequence of the following  facts: 
  \begin{enumerate}
  \item[(a)]
   By definition of $\cU(k)$, every $k$-elementary periodic point in $X_F$ has period $n$.
   \item[(b)]
    Every periodic orbit can be written as a finite concatenation of $k$-elementary periodic orbits.
    \item[(c)]
     The shift  $X_F$ is given by the infinite $g$-admissible concatenations of generating sequences of $k$-elementary periodic orbits in $X_F$ (see Section 2.4.).
     \end{enumerate}
      It follows from (c) that in order to prove the proposition, it suffices to show that $X_F$ contains one and only one $k$-elementary periodic orbit.
Let $x,y\in \EPer^k(f)$ with $\rv(\mu_x)=\rv(\mu_y)=w$. By (a), both $x$ and $y$ have period $n$. We have to show that $x$ and $y$ belong to the same periodic orbit. Assume on the contrary that $x$ and $y$ have distinct $k$-permutable orbits. Let $\tau_x=x_1 \cdots x_n$ and $\tau_y=y_1\cdots y_n$ be the generating sequences of $x$ and $y$ respectively. By replacing $x$ with an iterate of $x$ if necessary, we may assume $x_1=y_1$.  Thus, the cylinders $x_n$ and $y_n$ can both be followed by the cylinder  $x_1$. Define 
$l=\max\{\iota\in\{2,\dots,n\}: x_\iota\not=y_\iota\}.$ It follows that there exist unique $2\leq i,j\leq l-1$ such that $x_i=y_l$ and $y_j=x_l$. That is,
\begin{equation}\label{eq23}
\begin{split}
\tau_x=x_1 x_2\cdots x_{i-1} y_l x_{i+1} \cdots x_{l-1} x_lx_{l+1}\cdots x_n,& \,\, \,\,  {\rm and} \\
\tau_y=x_1 y_2\cdots y_{j-1}  x_l y_{j+1}\cdots y_{l-1} y_l x_{l+1}\cdots x_n.& \
\end{split}
\end{equation}
We define
\begin{equation}\label{eq231}
\begin{split}
a=x_1 x_2 \cdots x_{i-1},\, b=y_l x_{i+1} \cdots x_{l-1},\, c=x_l x_{l+1} \cdots x_n,&\, \, \,{\rm and}\\ 
a'= x_1 y_2 \cdots y_{j-1},\, b'= x_l y_{j+1}\cdots y_{l-1},\, c'=y_l x_{l+1} \cdots x_n.&\
\end{split}
\end{equation}
Hence $\tau_x=abc$ and $\tau_y=a'b'c'$. By Equations \eqref{eq23}, \eqref{eq231}, $bb', ac'$ and $a'c$ are generating sequences of periodic points of $g$. We write $\xi_1=\Or(bb'), \xi_2=\Or(ac')$ and $\xi_3= \Or(a'c)$ (see Figure 1.) and denote by $n_1,n_2,n_3$ the periods of the periodic points $\xi_1,\xi_2,\xi_3$  respectively.
It follows from the observation that $\xi_1\not\in X_F$ since $x_1\not\in bb'$. Similarly, $\xi_2, \xi_3\not\in X_F$
since $x_l\not\in ac'$ and $y_l\not\in a'c$.  
Let $\alpha\in S^{m-1}$ be a direction vector associated with $F$. 
We define the one-dimensional function $\phi=\alpha\cdot \Phi$. It follows that $\alpha\cdot w=\max \R(\phi)$, where $\R(\phi)$ denotes the rotation set of $\phi$ which is a compact interval in $\bR$. Moreover,  $F_1=\{\alpha\cdot w\}$ is a face of $\R(\phi)$ with $X_{F_1}=X_F$. Since $\xi_1,\xi_2,\xi_3\not\in X_{F}$ we obtain
\begin{equation}\label{eq232}
\mu_{\xi_1}(\phi), \mu_{\xi_2}(\phi), \mu_{\xi_3}(\phi) < \alpha\cdot w.
\end{equation}
Note that the integrals in Equation \eqref{eq232} are averages of the values of $\phi$ on the corresponding cylinders. Together with $x=abc$ and $y=a'b'c'$ we may conclude that
\begin{equation}
\begin{split}
WA\eqdef&
\frac{1}{2n}\left(n_1 \mu_{\xi_1}(\phi)+ n_2\mu_{\xi_2}(\phi)+n_3 \mu_{\xi_3}(\phi)\right)\\
=&\frac{1}{2n}\sum_{\iota=1}^n\left(\phi(x_\iota)+\phi(y_\iota)\right)=\alpha\cdot w.
\end{split}
\end{equation}
Note that $WA$ is  the weighted average of the integrals $\mu_{\xi_\iota}(\phi), \iota=1,2,3$. Thus, by Equation \eqref{eq232} we must have $WA<\alpha\cdot w$ which gives a contradiction. Thus, $x$ and $y$ must belong to the same periodic orbit and the proof is complete.
\end{proof}

\begin{center}
\begin{pspicture}(2,0)(-2,3)
\psset{unit=2cm}

\pscircle(-2,1){0.02}
\rput(-1.5,1){$\dots$}
\pscircle(-1,1){0.02}

\pscircle(-0.5,1){0.02}

\pscircle(0.5,1){0.02}
\rput(0,1){$\dots$}
\pscircle(0.5,1){0.02}

\pscircle(1,1){0.02}
\rput(1.5,1){$\dots$}
\pscircle(2,1){0.02}

\rput(-0.6,0){$($}\rput(-0.6,1){$($}
\rput(-2.1,0){$($}\rput(-2.1,1){$($}
\rput(0.3,0){$)$}\rput(0.3,1){$)$}
\psline(1,0.1)(1,0.9)
\psline(1.03,0.1)(1.03,0.9)
\psline(2,0.1)(2,0.9)
\psline(2.03,0.1)(2.03,0.9)

\psline(-0.5,0.1)(0.5,0.9)
\psline(-0.46,0.1)(0.54,0.9)
\psline(0.5,0.1)(-0.5,0.9)
\psline(0.54,0.1)(-0.46,0.9)

\psline(-2,0.1)(-2,0.9)
\psline(-1.97,0.1)(-1.97,0.9)

\rput(0.4,0){$($}\rput(0.4,1){$($}
\rput(-0.9,0){$)$}\rput(-0.9,1){$)$}
\rput(2.1,0){$)$}\rput(2.1,1){$)$}

\pscircle(-2,0){0.02}
\rput(-1.5,0){$\dots$}
\pscircle(-1,0){0.02}

\rput(-2.5,1){$\tau_x:$}
\rput(-2.5,0){$\tau_y:$}
\pscircle(-0.5,0){0.02}

\pscircle(0.5,0){0.02}
\rput(0,0){$\dots$}
\pscircle(0.5,0){0.02}

\pscircle(1,0){0.02}
\rput(1.5,0){$\dots$}
\pscircle(2,0){0.02}

\rput(-1.5,1.2){$a$}
\rput(-1.5,-0.2){$a'$}

\rput(-0.1,1.2){$b$}
\rput(-0.1,-0.2){$b'$}

\rput(1.5,1.2){$c$}
\rput(1.5,-0.2){$c'$}
\rput(-0.02,-0.8){{\bf Figure 1.} Periodic points $\xi_1=O(bb'), \xi_2=O(ac'), \xi_3=O(a'c)$ with other}
\rput(-0.5,-1.05){ periods  stemming from two permutable periodic points $x$ and $y$.}
\end{pspicture}
\end{center}
\vspace{2.8cm}
\noindent
For $r=1,\dots,m$ we define $$
LC_{k,r}(X,\bR^m)=\{\Phi\in LC_k(X,\bR^m): \dim \R(\Phi)=r\}
$$
and $\cU(k,r)= \cU(k)\cap LC_{k,r}(X,\bR^m)$.
The next result shows that $\cU(k,r)$ is open and dense in $LC_{k,r}(X,\bR^m)$.

\begin{theorem} \label{dichotomy}The following properties hold:
\begin{enumerate}
\item[(a)] $\cU(k,r)$ is a  nonempty open set in $LC_{k,r}(X,\bR^m)$.
\item[(b)] $LC_{k,r}(X,\bR^m)=\overline{\cU(k,r)}$.
\item[(c)] For every $\Phi\in \cU(k)$ there exists an open and dense set $D_\Phi\subset S^{m-1}$ of direction vectors such that for each $\alpha\in D_\Phi$ the zero-temperature measure $\mu_{\infty,\alpha\cdot \Phi}$ is supported on a $k$-elementary periodic orbit. 
\end{enumerate}
\end{theorem}

\begin{proof}
In the proof we will make use of the following elementary\\ 
Observation 1: If $\mu\in \cM$ and $\Phi,\tilde{\Phi}\in C(X,\bR^m)$ with $\|\Phi- \tilde{\Phi}\|<\varepsilon$ , then $\|\rv_\Phi(\mu)-\rv_{\tilde{\Phi}}(\mu)\|<\epsilon$.\\
First we prove that $\cU(k,r)$ is open in $LC_{k,r}(X,\bR^m)$. Let $\Phi\in\cU(k,r)$. 
Without loss of generality we may assume that $\R(\Phi)$ has nonempty interior (i.e. $r=m$) because otherwise we 
can consider the relative interior and relative boundary of $\R(\Phi)$.   We leave the elementary adjustments for the case $r<m$ to the reader. 
Thus, we will omit the parameter $r$ in the following.


Let $w$ be a vertex  of $\R(\Phi)$ and let $$E_{\rm per}(w)=\{x\in \EPer^k(f): \rv(\mu_x)=w\}.$$ By Proposition \ref{propkey},
 $$E_{\rm per}(w)=\{x,f(x),\dots,f^{n-1}(x)\},$$ for some  $k$-elementary periodic point $x$ with period $n$.
 
Claim 1: There exists  $\varepsilon_w>0$, such that if $ \tilde{\Phi}\in LC_k(X,\bR^m)$ with $\|\Phi-\tilde{\Phi}\|<\varepsilon_w$ then $\tilde{w}=\rv_{\tilde{\Phi}}(\{\mu_x: x\in E_{\rm per}(w)\})$ is a 
vertex of $\R(\tilde{\Phi})$. Let $H=H(w)$ be a supporting hyperplane of the face $\{w\}$
of $\R(\Phi)$. It follows from the definition that for all $y\in \EPer^k(f)\setminus E_{\rm per}(w)$ the rotation vectors $\rv(\mu_y)$ lie in the same open half space determined by $H$. Moreover, since
$\EPer^k(f)$ is finite, we have
$$
\varepsilon_w=\frac{\min\{{\rm dist}(\rv(\mu_y),H): y\in  \EPer^k(f)\setminus E_{\rm per}(w)\}}{2}>0.
$$
Let now $\|\Phi- \tilde{\Phi}\|<\varepsilon_w$. By applying Observation 1 we conclude that $\tilde{H}=(\tilde{w}-w)+H$ is a supporting hyperplane of the face $\{\tilde{w}\}$ of $\conv(\{\rv_{\tilde{\Phi}}(\mu_y): y\in \EPer(f)\})$. Therefore, Claim 1 follows from Theorem \ref{thmelemtentary}. Define $E_{\rm per}(\Phi)=\bigcup E_{\rm per}(w)$, where the union is taken over all vertices $w$ of $\R(\Phi)$.

Claim 2: There exists  $\varepsilon_0>0$ such that if $ \tilde{\Phi}\in LC_k(X,\bR^m)$ with $\|\Phi-\tilde{\Phi}\|<\varepsilon_0$, then $\rv_{\tilde{\Phi}}(\{\mu_y: y\in \EPer^k(f)\setminus E_{\rm per}(\Phi)\})\subset \inn \R(\tilde{\Phi})$. Pick $\varepsilon_0>0$ such that $B(\rv_\Phi(\mu_y), 2\varepsilon_0)\subset \inn\R(\Phi)$ for all $y\in \EPer^k(f)\setminus E_{\rm per}(\Phi)$. Consider $ \tilde{\Phi}\in LC_k(X,\bR^m)$ with $\|\Phi-\tilde{\Phi}\|<\varepsilon_0$. Since $\R(\Phi)=\conv(\{\rv_\Phi(\mu_y): y\in E_{\rm per}(\Phi)\})$, Observation 1 implies that the Hausdorff distance of $\partial \R(\Phi)$ and $\partial \conv(\{\rv_{\tilde{\Phi}}(\mu_y): y\in E_{\rm per}(\Phi)\})$ is smaller than $\varepsilon_0$. Moreover, Observation 1 yields that  $\rv_{\tilde{\Phi}}(\mu_y)\in B(\rv_\Phi(\mu_y), \varepsilon_0)$ for all $y\in \EPer^k(f)\setminus E_{\rm per}(\Phi)$. Together this implies that  $\rv_{\tilde{\Phi}}(\{\mu_y: y\in \EPer^k(f)\setminus E_{\rm per}(\Phi)\})\subset \inn \conv(\{\rv_{\tilde{\Phi}}(\mu_x): x\in E_{\rm per}(\Phi)\})$ which implies Claim 2.

Finally we define $\varepsilon=\min \left(\{\varepsilon_0\}\cup \{\varepsilon_w: w\,\, {\rm vertex\,\, of}\, \,\R(\Phi)\}\right).$ We now apply Claim 1 and Claim 2  and conclude that if $ \tilde{\Phi}\in LC_k(X,\bR^m)$ with $\|\Phi-\tilde{\Phi}\|<\varepsilon$ then $\tilde{\Phi}\in \cU(k)$. Hence, $\cU(k)$ is open. The statement that $\cU(k)$ is nonempty will follow from part $(b)$.

Next we prove $LC_{k,r}(X,\mathbb{R}^m)=\overline{\cU(k,r)}$. As before we assume $r=m$ and omit the parameter $r$ in the following. Let $\Phi\in LC_k(X,\mathbb{R}^m)$. Suppose there exists a non-vertex point $w\in \partial \R(\Phi)$ with at least one $k$-elementary periodic measure in the rotation class of $w$. Pick $x\in \EPer^k_n(f)$ with $\rv(\mu_x)=w$ which is minimal in the following sense: There is no elementary periodic orbit $y$ with $\rv(\mu_y)=w$ whose cylinder set associated with  $y$ (cf Definition \ref{defkelementary}) is a proper subset of the cylinder set $\{\cC_i=\cC(f^i(x)): i=0,\dots,n-1\}$ of the generating sequence of $x$. Clearly, such a minimal $x$ exists. Let $\alpha\in S^{m-1}$ be a direction vector of the largest (in the sense of inclusion) face containing $w$.
Further, let $\chi_x$ denote the characteristic function of $\cC_0\cup\dots\cup \cC_{n-1}$ on $X$.   
We define a perturbation of $\Phi$ along the direction vector $\alpha$. More precisely,  we define 
\begin{equation}\label{pertub}\Phi_t=\Phi+  (t\alpha_1\chi_x,\dots, t\alpha_m\chi_x).
\end{equation}
 Clearly, $\Phi_t  \rightarrow \Phi$
 as
$t\rightarrow 0$. It follows that
the corresponding rotation vector 
$w+t\alpha $ is a vertex point of $\R(\tilde{\Phi}_t)$ for any $t>0$. Moreover, the vertex point $w+t\alpha $ satisfies Property (a) of Definition \ref{defUk}.
Therefore, by applying the argument in the proof of Proposition \ref{propkey} to the face $\{w+ t\alpha\} $ we may conclude that $w+ t\alpha$ has only one $k$-elementary periodic point measure in its $\Phi_t$ rotation class.  (See Figure 2 below). 
Using induction and performing this perturbation one by one (and arbitrarily  small) we can remove all non-vertex rotation vectors of elementary periodic point measures from the boundary of $\R(\Phi)$ by an arbitrary small perturbation. To complete the proof of $(b)$ we still have to remove all but one $k$-elementary periodic point measures from the rotation classes of the other vertex points. However, this can be accomplished by performing a similar perturbation as in  \eqref{pertub}, and again doing it one by one, that is, vertex by vertex. We leave the details to the reader.

Finally, we show Assertion (c). Let $\Phi\in \cU(k)$. 
By Proposition \ref{propkey},  for any direction vector $\alpha\in S^{m-1}$  of any of the vertices $w$ of $\text{Rot}(\Phi)$, there is one and only one $k$-elementary periodic point measure in its rotation class $\{\mu\in\mathcal{M}: \text{rv}(\mu)=\omega\}$. By Theorems \ref{thmfaceshift} and  \ref{KW}, the zero-temperature measure $\mu_{\infty,\alpha\cdot \Phi}$ is supported on this $k$-elementary periodic orbit.  Since $\text{Rot}(\Phi)$ is a polyhedron, the direction vectors associated with the vertices of $\R(\Phi)$ form an open and dense subset of $S^{m-1}$.  This completes the proof of the theorem.
\end{proof}

\begin{center}
\begin{pspicture}(-3,-2)(3,2)
\psset{unit=0.4cm}
\psline{->}(-3.5,1.25)(-2,3)
\rput(-2.8,3.2){$\alpha$}
\psline(-8,0)(-5,2.5)(-2,0)(-3.5,-3)(-6.5,-3)(-8,0)
\pscurve{->}(-0.7,3.2)(0.7,4)(1.8,3.2)
\rput(-2.6,1.4){$w$}
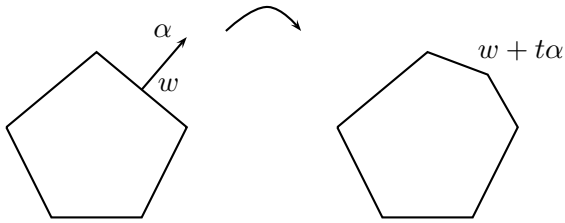
\rput(0.05,-4.8){{\bf Figure 2.} Perturbation to remove  non-vertex elementary periodic points.}
\psline(3,0)(6,2.5)(8,1.75)(9,0)(7.5,-3)(4.5,-3)(3,0)
\rput(9.1,2.55){$w+t\alpha$}
\end{pspicture}
\end{center}


\section{The proof of the Theorem A}

This section is devoted to the proof of  our main result Theorem A.  Our approach uses  a suitable higher dimensional function $\Phi$ that encodes all one-dimensional functions. We will then obtain Theorem A by applying results about the connection between rotation sets and zero-temperature measures (see Theorem \ref{KW})  to the function $\Phi$.   We note that although $\Phi$ is far from being unique, we will call $\Phi$ a universal function. 


We start by presenting a version of Theorem A for higher-dimensional functions. Recall that $m_c(k)$ denotes the cardinality of the set of cylinders of length $k$ in $X$.

\begin{theorem}\label{finite measures}Let $k\in\bN$. Then 
\begin{enumerate}
\item[\rm (a)]
For  $\Phi\in LC_k(X,\bR^m)$  there exist a finite set $\mathcal{M}(\Phi)\subset \mathcal{M}_E$ such that  for each $\alpha\in S^{m-1}$ the zero-temperature measure $\mu_{\infty,\alpha\cdot \Phi}$ is a convex sum of measures in $\mathcal{M}(\Phi)$.
\item[\rm(b)]
There exists $\Phi_{\rm univ}\in LC_k(X,\bR^{m_c(k)})$
and an analytic surjection $I:LC_k(X,\bR)\setminus \{\phi\equiv 0\}\to S^{m_c(k)-1}$  such that for all $\phi\in LC_k(X,\bR)\setminus \{\phi\equiv 0\}$ we have that
$\mu_{\infty,\phi}=\mu_{\infty,I(\phi)\cdot \Phi_{\rm univ}}$ is a convex sum of measures in  $\cM(\Phi_{\rm univ})$. 
\end{enumerate}
\end{theorem}
\begin{proof}
We first prove {\rm (a).} Let $k\in \bN$ and $\Phi\in LC_k(X,\bR^m)$. It follows from  Theorem \ref{thmelemtentary} that the rotation set $\R(\Phi)$ has finitely many faces $F_1,\dots, F_\ell$. Therefore,
\begin{equation}\label{parts1}
S^{m-1}=\bigcup_{i=1}^{\ell} \alpha(F_i)
\end{equation}
is a partition of $S^{m-1}$ into   disjoint nonempty sets of direction vectors. Let us consider a fixed face $F\in \{F_1,\dots,F_\ell\}$ and let  $\alpha\in \alpha(F)$.
By Theorem \ref{KW},  the rotation vector of the zero-temperature measure $\mu_{\infty,\alpha\cdot \Phi}$ lies in the face  $F=F_{\alpha}(\Phi)$,
and 
\begin{equation}
h_{\mu_{\infty,\alpha\cdot \Phi}}(f)=\sup\{h_{\nu}(f):\text{ rv}(\nu)\in F\}.
\end{equation}
On the other hand, by Theorem \ref{thmfaceshift}, $\text{rv}(\nu)\in F$ if and only if $\text{supp}\, \nu\in X_{F},$
where $X_{F}=\overline{\{x\in \Per(f): \rv(\mu_x)\in F\}}.$ 
It follows that $\mu_{\infty,\alpha\cdot \Phi}$ is a measure of maximal entropy of  $f\vert_{X_F}$.
By Equation \eqref{sumtrans}, there exist $t=t(F)\in \bN$ and invariant sets $X_1,\dots,X_t\subset X$ such that
$f\vert_{X_j}$ is a transitive subshift of finite type and
$X_{F}=X_{1}\cup\dots \cup X_{t}.$
 It is well known that for each $j=1,\dots,t$ there exists a unique measure of  maximal entropy  $\mu_{F,j}$ of
$f\vert_{X_j}$ (which is called the Parry measure of 
$f\vert_{X_j}$).  We conclude that
\begin{eqnarray*}h_{\text{top}}(X_{F})&=&\max\{h_{\text{top}}(f\vert_{X_1}), \dots, h_{\text{top}}(f\vert_ {X_{t}})\}\\
&=&\max\{h_{\mu_{F,1}}(f),\dots,h_{\mu_{F,t}}(f)\}.
\end{eqnarray*}
Define  
\begin{equation}\label{eqmf}
\mathcal{M}_{F}=\{\mu_{F,j}: h_{\mu_{F,j}}(f)=h_{\text{top}}(X_{F})\},
\end{equation}
which is  a nonempty finite set of ergodic measures. Define $n_F=\vert \cM_F \vert$.

Using that the entropy is affine on $\cM$ we conclude that  any  measure of maximal entropy of $f\vert_{X_{F}}$  
must be a convex combination of  measures in $\mathcal{M}_{F}.$ 
In particular, for any $\alpha\in \alpha(F)$ there exist $0\leq a_{1,\alpha},\dots,a_{n_F,\alpha}\leq 1$ with $\sum a_{j,\alpha}=1$ such that
\begin{equation}\label{eqconvmax}
\mu_{\infty,\alpha\cdot \Phi}=\sum_{j=1}^{n_F} a_{j,\alpha} \mu_{F,j},
\end{equation}
where $\mu_{F,j}\in \cM_F$.
We define 
$
\mathcal{M}(\Phi)=\bigcup_{i=1}^\ell \mathcal{M}_{F_i}.
$
It now follows from Equations \eqref{parts1}, \eqref{eqmf} and \eqref{eqconvmax} that $\mathcal{M}(\Phi)$ is a  finite set of  ergodic measures with the desired property.
This completes the proof of part (a). \\
Next we prove (b). Set $m=m_c(k)$. Let $\{\cC_k(1),\dots,\cC_k({m})\}$ denote the set of cylinders of length $k$ in $X$. Recall that $LC_k(X,\bR)$ is identified with $\bR^m$.
Let $\cB=\{v_1,\dots,v_m\}$ be a fixed basis of $\bR^m$. We define
$\Phi_{\rm univ}:X\to \bR^m$  by $\Phi_{\rm univ}\vert_{\cC_k(i)}=v_{i}$. Clearly,  $\Phi_{\rm univ}\in LC_k(X,\bR^m)$. 
Since $\cB$ is a basis of $\bR^m$ there exists an isomorphism $L:LC_k(X,\bR)\to \bR^m$ such that $L(\phi)_1 v_1+\dots+ L(\phi)_mv_m=\phi$. We define 
$$
 I:LC_k(X,\bR)\setminus \{0\}\to S^{m-1},\quad I(\phi)=\frac{L(\phi)}{\| L (\phi)\|}.
 $$
 Clearly, $I$ is an analytic surjection. Moreover, for all  $\phi\in LC_k(X,\bR)\setminus \{0\}$ we have $\phi=\|L(\phi)\| I(\phi)\cdot \Phi$. We conclude from part (a) that $\mu_{\infty,\phi}=\mu_{\infty,I(\phi)\cdot \Phi_{\rm univ}}$ is a convex sum of measures in $\cM(\Phi_{\rm univ})$.
 \end{proof}
 \begin{center}
\begin{pspicture}(-3,-4)(0,2.5)
\psset{unit=0.4cm}
\psline{->}(-5,-0.5)(-1.8,2.6)
\psline{->}(-5,-0.5)(-8,2.8)
\psline{->}(-5,-0.5)(-5,-5)
\psline{->}(-5,-0.5)(-1,-2.5)
\rput(-5,1.2){\tiny $\text{Rot}(\Phi)$}
\psline{->}(-5,-0.5)(-9,-2.5)
\rput(-5, 4.5){$\cU_6$}
\rput(0.2,0){$\cU_7$}
\rput(-2.3,-4.95){$\cU_8$}
\rput(-7.5,-4.8){$\cU_9$}
\rput(-10.3,0.2){$\cU_{10}$}
\rput(-8.5,3){$\alpha_1$}
\rput(-5,-5.4){$\alpha_4$}
\rput(-0.4,-2.8){$\alpha_3$}
\rput(-1.2,2.7){$\alpha_2$}
\rput(-9.5,-2.6){$\alpha_5$}
\rput(1,3){$S^{m-1}$}
\psline(-8,0)(-5,2.5)(-2,0)(-3.5,-3)(-6.5,-3)(-8,0)
\pscircle(-5,-0.5){4.5}
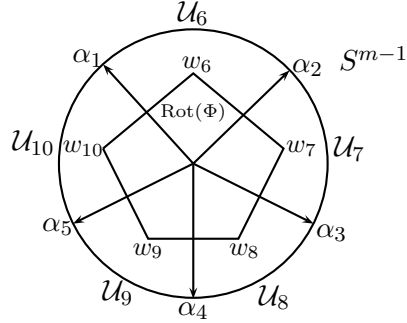
\rput(-4.6,-8){{\bf Figure 3.} The partition of the sphere for a rotation set with ten faces.}
\rput(-8.65,0){\small $w_{10}$}
\rput(-4.87,2.8){\small $w_6$}
\rput(-1.4,0){\small $w_7$}
\rput(-3.3,-3.4){\small $w_8$}
\rput(-6.5,-3.4){\small $w_9$}
\end{pspicture}
\end{center}

\begin{remark}\label{remdeterminant}
We note that the universal function $\Phi_{\rm univ}$ in part (b) of Theorem \ref{finite measures} is not unique. On the contrary, any function $\Phi\in LC_k(X,\bR^{m_c(k)})$ whose values on the cylinders of length $k$ form
a basis in $\bR^{m_c(k)}$ has this property. Thus, the set of $\Phi$ with the universal function property is a dense open subset of  $LC_k(X,\bR^{m_c(k)})$.
\end{remark}
By putting together the results of Proposition \ref{propkey} and Theorems \ref{dichotomy} and  \ref{finite measures} we are finally ready to prove Theorem A.
\begin{proof}[The proof of Theorem A]
Fix $k\in\mathbb{N}$ and let $m=m_c(k)$. It follows from Theorem \ref{finite measures} (b) and Remark \ref{remdeterminant} that there exists a universal function $\Phi_{\rm univ}\in LC_k(X,\bR^m)$ such that $\Phi_{\rm univ}\in \cU(k)$. Since $\Phi_{\rm univ}$ is universal there exists $\alpha\in S^{m-1}$ such that $\alpha\cdot \Phi_{\rm univ}\equiv const$. Thus,  by \cite[Proposition 6]{KW1},   $\inn \R(\Phi_{\rm univ})=\emptyset$ which implies $\dim \R(\Phi_{\rm univ})<m$. 
We conclude that $\R(\Phi_{\rm univ})$ is a face of itself.
Let $F_1,\dots, F_N$ denote the  faces of $\R( \Phi_{\rm univ})$. We recall  the definition of the subshift of finite type $X_F$ associated with a face $F$, see \eqref{sumtrans}.
We arrange the faces   $F_1,\dots, F_{\ell_1},\dots,F_{\ell_2},\dots,F_N$ such that $F_{1},\dots,F_{\ell_1}$ correspond to the vertex point faces of $\R( \Phi_{\rm univ})$ and
 $F_{\ell_1+1},\dots, F_{\ell_2}$ are the faces that correspond to the non-discrete transitive subshifts of finite type $X_{F_i}$.
In particular, we define $F_{\ell_1+1}=\R( \Phi_{\rm univ})$ which implies $X_{F_{\ell_1+1}}=X$.  We define $\cU_{\ell_1+1}=\{\phi\,\, \text{cohomologous to zero}\}$ and 
  $\cM_{\ell_1+1}=\{\mu_{\rm mme}\}$, where $\mu_{\rm mme}$ is the unique measure of maximal entropy of $f$. 
 It follows from \eqref{eqLiv} that $\cU_{\ell_1+1}$ is a subspace of  $\bR^m$. In particular $\cU_{\ell_1+1}$ is a convex cone. 
 We conclude that statement (b) holds for $i=\ell_1+1$.
For $i=1,\dots,N$ with $i\not=\ell_1+1$ we define
\begin{equation}\label{defcU}
\cU_i=\bigcup_{\alpha\in \alpha(F_i)} \{t \alpha: t\in \bR^+\}.
\end{equation}
Clearly, $\{\cU_1,\dots,\cU_N\}$ is a partition of $\bR^{m}$. (See Figure 3.) Further, since the $\alpha(F_i)$'s are convex sets\footnote{To make this convexity argument we consider here a slightly more general definition for $\alpha(F)$. Namely, we allow direction vectors $\alpha$ with $\|\alpha\|\not=1$.}, we may conclude that the sets $\cU_i$ are convex cones.  Assertion (a) is now a consequence of Proposition \ref{propkey} and Theorem \ref{dichotomy} (b) and (c) since $\Phi_{\rm univ}\in \cU(k)$. 
Further, assertion (b)  follows from Theorem \ref{finite measures}, Equation \eqref{eqconvmax} and the fact that $n_i=1$ for $i=\ell_1+1,\dots,\ell_2$. 
Finally, we consider assertion (c). Equation \eqref{messum} follows from Theorem \ref{finite measures} and Equation \eqref{eqconvmax}. The claim that for each $\cU_i$ there are only finitely many
choices for the coefficients $a_{j_i,\phi}$ is a consequence \cite[Theorem 2.1 (2)]{Br}.
Finally, that the entropy is constant on $\cM_i$ follows from Equations \eqref{eqmf} and \eqref{eqconvmax}. 
The proof of  Theorem A is complete. 
\end{proof}

\section{Regularity of localized entropy function on the boundary
}
The goal in this section is to study the regularity of $w\mapsto \mathcal{H}(w)$ on the boundary of the rotation set for locally constant functions. Let $f:X\to X$ be a transitive subshift of finite type with transition matrix $A$ and let $\Phi\in LC_k(X,\mathbb{R}^m)$. Applying  Theorem B in \cite{KW1} to our situation gives the following.

\begin{theorem}[\cite{GKLM, KW1}]\label{thmKW1}
 Let $\Phi\in LC_k(X,\mathbb{R}^m)$ with ${\rm int}\, \R(\Phi)\not=\emptyset$. Then
 \begin{enumerate}
\item[(i)] The map $T_\Phi:\bR^m\to {\rm int}\, \R(\Phi)$ defined by $v\mapsto \rv(\mu_{v\cdot \Phi})$ is a real-analytic diffeomorphism.
\item[(ii)] For all $v\in \bR^m$, the measure $\mu_{v\cdot \Phi}$ is the unique measure that maximizes entropy among the measures  $\cM_\Phi(T_\Phi(v))$. In this case we say that $\mu_{v\cdot \Phi}$ is the unique localized measure of maximal entropy at $w=T_\Phi(v)$.
\item[(iii)]  The map $w\mapsto \mathcal{H}(w)$ is real-analytic on $\inn  \R(\Phi)$.
\end{enumerate}
 \end{theorem}
 More precisely, Theorem \ref{thmKW1} was proved in \cite{KW1} for so-called STP maps (which includes subshifts of finite type) and H\"older continuous functions. In \cite{GKLM} the authors give an alternate proof for Theorem \ref{thmKW1} and provided a technique which allows to derive the analyticity of  map $w\mapsto \mathcal{H}(w)$ also in the case $\inn \R(\Phi)=\emptyset$ but now on the relative interior of $\R(\Phi)$. Roughly speaking, in order to treat the empty interior case one identifies cohomologous functions in the family $\{v\cdot\Phi:v\in \bR^m\}$.
 
 Since the faces of the rotation set of  locally constant functions are themselves polyhedra, one might expect from Theorem \ref{thmKW1} that the localized entropy function restricted to the relative interior of such a face is also analytic. However, Example \ref{example3} in Appendix B  shows that $\cH$ is in general not analytic on the relative  interior of a face.
 
Let $F$ be a face of $\R(\Phi)$. It follows that $F$ is also a polyhedron. Let $d(F)<m$ denote the dimension of $F$. Since we are interested in the regularity of the localized entropy function on the relative interior of $F$ we may assume that $F$ is not a singleton, i.e. $d(F)>0$. Let $X_F=X_1\cup \dots\cup X_t$ be as in Equations \eqref{perface} and \eqref{sumtrans}. It follows that $f_i=f|_{X_i}$ is a transitive subshift of finite type and that $F_i\eqdef \R(\Phi,f_i)$ is a polyhedron of dimension $n_i\leq d(F)$. Moreover, $F=\conv(F_1\cup\dots\cup F_t)$. Further,  Theorem \ref{thmKW1} shows that $\mathcal{H}_i\eqdef \mathcal{H}_{\Phi,f_i}$ is real-analytic on $\ri F_i$.

Let us recall the following definition. Let $D_1,\dots,D_t$ be non-empty bounded subsets of $\bR^m$, let $h_i:D_i\to \bR$ be bounded continuous functions, and let $D=\conv(D_1\cup \dots \cup D_t)$. We say $h:D\to \bR$ is the \emph{concave envelope}\footnote{We note that the concave envelope is the analog to the convex hull of a family of functions. The latter is more commonly studied in convex analysis, see e.g. \cite{R}.}
  of $h_1,\dots,h_t$ and write $h=\Ec(h_1,\dots, h_t)$ if $h$ is the smallest concave function satisfying $h|_{D_i}\geq h_i$ for all $i=1,\dots,t$. 
We have the following.
\begin{proposition}\label{concenv}
Let $f:X\to X$ be a transitive subshift of finite type and let $\Phi\in LC_k(X,\mathbb{R}^m)$. Let $F$ be a face of $\R(\Phi)$ and let $X_F=X_1\cup \dots\cup X_t$ be the decomposition of $X_F$ into transitive subshifts $X_i$ of finite type associated with $F$ (see \eqref{perface} and \eqref{sumtrans}). Then $\mathcal{H}|_F$ is the concave envelope  of
$\cH_1,\dots,\cH_t$ where $\mathcal{H}_i= \mathcal{H}_{\Phi,f|_{X_i}}$.
\end{proposition}
\begin{proof}
Let $h=\Ec(\cH_1,\dots,\cH_t)$.
For $i=1,\dots,t$ let $F_i\eqdef\R(\Phi|_{X_i})$. Further, let $\EPer(F)=\{x\in \EPer(f): \rv(\mu_x)\in F\}$. It is a consequence of Theorem \ref{thmelemtentary} that $F=\conv(\rv(\EPer(F)))$. Since $X_F=X_1\cup \dots\cup X_t$
we have $\EPer(F)\subset X_1\cup \dots\cup X_t$. We conclude that $F\subset \conv(F_1\cup\dots\cup F_t)$. On the other hand, we
clearly have $F_i\subset F$ for all $i=1,\dots,t$ and since $F$ is convex we obtain that  $\conv(F_1\cup\dots\cup F_t)\subset F$. Hence, $h$ and $\cH|_F$ are both functions defined on $F$.

We need to show that $h=\cH|_F$. Since $\cH$ and therefore  in particular  $\cH|_F$ is a concave function, the statement $h\leq \cH|_F$ follows directly from the definition of the concave envelope. Let $w\in F$. We claim that $\cH(w)\leq h(w)$. Since the entropy map $\nu\mapsto h_{\nu}(f)$ is upper semi-continuous there exists $\mu\in 
\cM_{f|_{X_F}}$ with $h_\mu(f)=\cH(w)$ and $\rv(\mu)=w$. Using that  $X_F=X_1\cup \dots\cup X_t$ there exist $\mu_1,\dots,\mu_t$ with $\mu_i\in \cM_{f\vert_{X_i}}$,
and $\lambda_1,\dots,\lambda_t\geq 0$ with $\sum_{i=1}^t \lambda_i =1$
such that $\mu= \sum_{i=1}^t \lambda_i \mu_i$.  By definition, $\rv(\mu_i)\in F_i$ and $h_{\mu_i}(f)\leq\cH_i(\rv(\mu_i))$. Note that $w=\sum_{i=1}^t \lambda_i \rv(\mu_i)$. Using that $h$ is a concave function on $F$ we may conclude that $h(w)\geq \sum_{i=1}^t \lambda_i \cH_i(\rv(\mu_i))$. On the other hand, the fact that the entropy is affine shows that $\cH(w)=h_{\mu}(f)=\sum_{i=1}^t \lambda_i h_{\mu_i}(f)$.
This proves the claim and  the proof of the proposition is complete.
\end{proof}
We recall that for $k\in \bN_0\cup\{\infty,\omega\}$ a function $h:[a,b]\to\bR$ is piecewise $C^k$ if  there exists $a=c_0<c_1<\dots <c_\ell=b$ such that $h|_{(c_i,c_{i+1})}$ is $C^k$ for all $i=0,\dots, \ell-1$. Let $\Phi\in LC_k(X,\mathbb{R}^2)$. We say $\cH$ is piecewise $C^k$ on the boundary of $\R(\Phi)$ if there exist finitely many points $a_0,\dots,a_{\ell}=a_0 \in\partial \R(\Phi)$  such that $[a_i,a_{i+1}]$ is a line segment contained in a face $F_i$ of $\R(\Phi)$ and
$\cH\vert_{(a_i,a_{i+1})}$ is $C^k$.

We will use the following elementary fact.
\begin{lemma}\label{lemelef}
For $i=1,\dots,t+1$ let $h_i:D_i\to \bR$  be bounded functions with bounded domains. Then $\Ec(h_1,\dots, h_{t+1})=\Ec(\Ec(h_1,\dots, h_t),h_{t+1})$.
\end{lemma}
The statement follows from an elementary induction argument and is left to the reader. Finally, we prove the following.

\begin{theorem}\label{thmpwanalytic}Let $f:X\to X$ be a transitive subshift of finite type and let
 $\Phi\in  LC_k(X,\mathbb{R}^2)$. Then the localized entropy function $\mathcal{H}$ is piecewise $C^1$ on  $\partial \R(\Phi)$.
Moreover, the non-differentiability points of $\cH|_{\partial \R(\Phi)}$  are rotation vectors of $k$-elementary periodic point measures.
\end{theorem}
\begin{proof}
Since $\partial \R(\Phi)$ is a polygon whose vertices are rotation vectors of elementary periodic point measures, it is sufficient to proof the statement for every face of $\R(\Phi)$. Let $F$ be such a face. If $F$ is a singleton then there is nothing to prove. Thus, we can assume
that $F$ is a non-trivial line segment.
Let $X_F=X_1\cup \dots \cup X_t$ be the decomposition of $X_F$ into transitive subshifts $X_i$ (see Equations \eqref{perface} and \eqref{sumtrans}). 
Further, let $F_i=\R(\Phi|_{X_i})\subset F$ and $\mathcal{H}_i= \mathcal{H}_{\Phi,f|_{X_i}}$.
We will prove the statement by induction for $t$.\\
If $t=1$ then $X_F$ is itself a transitive subshift of finite type. As a consequence of Theorem \ref{thmKW1} (see the discussion after the statement of Theorem \ref{thmKW1}) we obtain that $\cH|_F=\cH_1$ is real-analytic on the relative interior of $F$ and therefore in particular of class $C^1$.\\
$t\mapsto t+1:$  Applying Proposition \ref{concenv} and Lemma \ref{lemelef} we conclude that it is sufficient to consider the case $t=1$.   Assume that the assertions of the theorem hold for $\cH_1$. As before Theorem  \ref{thmKW1} implies that $\cH_2$ is real-analytic on  the relative interior of $F_2$. Moreover, the end points of the intervals $F_i$ are rotation vectors of $k$-elementary periodic point measures. Let $w_0\in \ri F$ such that $w_0$ does not coincide with a rotation vector of a $k$-elementary periodic point measure. In particular, $w_0$ is not an end point of the intervals $F_1$ and $F_2$. Since there are only finitely many $k$-elementary orbits, there exists a neighborhood $U\subset F$ of $w_0$ which does not contain a  rotation vector of a $k$-elementary periodic point measure. Hence, $\cH_i|_U\in C^1(U,\bR)$ for $i=1,2$. In what follows we only consider points in $U$. 
Further, whenever we refer to a line segment of the graph of $\cH$ we mean a non-trivial line segment of maximal length.
Our 
proof is based on the following two elementary observations.

Observation 1: If $(w_0,\cH(w_0))$ lies on a line segment of the graph of $\cH$ and $\cH_i(w_0)=\cH(w_0)$ then $D\cH_i(w_0)$ coincides with the slope of the line segment.

Observation 2: If $\cH(w_0)=\cH_1(w_0)=\cH_2(w_0)$ then $D\cH_1(w_0)=D\cH_2(w_0)$.\\
Observations 1 and 2 follow from the facts that $\cH_i|_U\in C^1(U,\bR)$ and that $\cH_i$ and $\cH$ are concave functions.   We leave the elementary proofs to the reader. To prove the theorem we need to show that $\cH|_F$ is $C^1$ at $w_0$. 
We shall consider several cases. \\
Case 1: The point $(w_0,\cH(w_0))$ belongs to the interior of a line segment of the graph of $\cH$. In this case $\cH$ is affine in a neighborhood of $w_0$ and is therefore  $C^1$.\\
Case 2: There exists a neighborhood $V\subset F$ of $w_0$ such that $\cH|_V=\cH_i|_V$ for some $i\in\{1,2\}$. This case is trivial since $\cH_i$ is $C^1$ on $U\cap V$.\\
 Case 3: The point $(w_0,\cH(w_0))$ is an  end point of two distinct line segments of the graph of $\cH$. Let $L_l$ denote the line segment on the left-hand side of $w_0$ and $L_r$ the line segment on the right-hand side of $w_0$.   We claim that this case can not occur. Otherwise, by maximality of the line segments (also using that $\cH$ is concave), the slope of $L_l$ must be strictly larger than the slope of $L_r$. We conclude that $(w_0,\cH(w_0))$ is an extreme point of the region bounded above by the graph of $\cH$.  Let $\mu\in \cM_{f|_{X_F}}$ with $\rv(\mu)=w_0$ and $h_\mu(f)=\cH(w_0)$, i.e., $\mu$ is a localized measure of maximal entropy at $w_0$. 
 It follows from an ergodic decomposition argument and the extreme point property of $(w_0,\cH(w_0))$ that we can assume that $\mu$ is ergodic. Hence, $\mu$ must be supported on $X_i$ for some $i=1,2$. We conclude that $\cH_i(w_0)=\cH(w_0)$. It now follows from Observation 1 that $D\cH_i(w)$ coincides with both, the slope of $L_l$ and
 the slope of $L_r$. But this is a contradiction and the claim is proven. \\
 Case 4: There is a neighborhood $V\subset F$ of $w$ on which the graph of $\cH$ coincides on one side of $w$ with a line segment and on the other side with the graph of $\cH_i$ for some $i=1,2$.  In this case $\cH$ restricted to $V\setminus\{w\}$  is clearly a $C^1$-function. Moreover, Observation 1 shows that $\cH$ is  $C^1$ on $V$.\\
 Case 5: $\cH_1(w_0)=\cH_2(w_0)=\cH(w_0)$ and neither of the Cases 1-4 applies. First note that Cases 1-5 actually cover all possibilities. Indeed, if $\cH(w)>\max\{\cH_1(w_0),\cH_2(w_0)\}$  then   $(w,\cH(w))$ must lie in the interior of a line segment of the graph of $\cH$ which is Case 1. Moreover, if $\cH_1(w_0)<\cH_2(w_0)=\cH(w_0)$ or $\cH_2(w_0)<\cH_1(w_0)=\cH(w_0)$ then one of the Cases 1,2 or 4 applies. Next we show that $\cH$ is 
 differentiable at $w_0$. By Observation 2,  $D\cH_1(w_0)=D\cH_2(w_0)$. Clearly, the previous cases fail either on the left-hand side of $w_0$ or on the right-hand of $w_0$ or on both sides. Without loss of generality we only consider the left-hand side case. The right-hand side case is analogous. Since Cases 1-4 do not apply in any neighborhood $V\subset F$ of $w_0$ there exist $\epsilon>0$ and a strictly increasing sequence $(w_n)_n\subset (w_0-\epsilon,w_0)$ with $\lim w_n =w_0$ such that $\cH(w_n)=\cH_1(w_n)$ for $n$ even and $\cH(w_n)=\cH_2(w_n)$ for $n$ odd. We note that this includes the possibility of $\cH_1(w_i)\not=\cH_2(w_i)$ for $i=n,n+1$ in which case
 the graph of  $\cH$ must contain a line segment between the points $(w_n,\cH(w_n))$ and $(w_{n+1},\cH(w_{n+1}))$.
 Let $w\in [w_1,w_0)$ and let $N\in \bN$ such that $w\in [w_N,w_{N+1})$. Using that $\cH$ is concave, we conclude 
 \begin{equation}\label{eqdiffquo}
 \frac{\cH(w)-\cH(w_0)}{w-w_0}\in \left[\frac{\cH(w_{N+1})-\cH(w_0)}{w_{N+1}-w_0}, \frac{\cH(w_{N})-\cH(w_0)}{w_{N}-w_0}\right].
 \end{equation}
 Clearly both difference quotients on the right-hand side of Equation \eqref{eqdiffquo} converge to $D\cH_1(w_0)=D\cH_2(w_0)$ as $N\to \infty$. This implies that 
 \[
 \lim_{w\to w_0-}  \frac{\cH(w)-\cH(w_0)}{w-w_0}=D\cH_i(w_0).\]
 It now follows from either applying the same argument to the right-hand side of $w_0$ or from Cases 1-4 that $\cH$ is differentiable at $w_0$ with
 $D\cH(w_0)=D\cH_1(w_0)=D\cH_2(w_0)$. Let $w\in V$. From previous arguments (including Cases 1-4) we know that $\cH$ is differentiable at $w$. Moreover, we either have $D\cH(w)=D\cH_i(w)$ for some $i=1,2$, or $(w,\cH(w))$ lies on a line segment of the graph of $\cH$ with one endpoint of this line segment being closer to $w_0$ as $w$. Finally, since $D\cH(w_0)=D\cH_1(w_0)=D\cH_2(w_0)$ we obtain $\lim_{w\to w_0}D\cH(w)=D\cH(w_0)$. Thus $\cH$ is $C^1$ near $w_0$.
 \end{proof}
 \begin{remark}
 One might ask if the proof of Theorem \ref{thmpwanalytic} can be extended to show that $\cH$ is piecewise analytic on $\partial \R(\Phi)$. In fact, piecewise analyticity holds provided the graph of $\cH$ contains only finitely many line segments. We are not aware of an obstruction to infinitely many lines segments. While we do not have an actual example, we believe that there do exist locally constant $2$-dimensional functions such that $\cH$ is not piecewise
 analytic on  $\partial \R(\Phi)$.   
\end{remark}
\appendix
\section{Applications of Theorem A}
In this appendix we apply Theorem A to various examples. In addition, we  use the results in Section 2.4  and the formula for the Parry measure, see e.g. \cite{PP}.
In order to keep the examples of reasonable size we consider potentials that are constant on cylinders of length $2$.  We use the notation $\phi\vert_{\cC_2(x)}=\phi_{x_1,x_2}$ and identify $\phi$ with a $d\times d$-matrix $M_\phi=(\phi_{i,j})_{0\leq i,j\leq d-1}$.
\subsection{The 2-symbol full-shift}
\begin{example}\label{example1} 
Let $X=\{0,1\}^\bN$ and let $f:X\to X$ be the shift map.  One can readily check that $\EPer^2(f)$ coincides with the orbits of $\Or(0), \Or(1), \Or(01), \Or(110),$ $\Or(001)$ and $\Or(0011)$. This leads to the following cases:
\begin{enumerate}
\item[\rm{(1)}] If $\phi_{0,0}>\max\{\phi_{1,1},\frac{1}{2}\left(\phi_{0,1}+\phi_{1,0}\right)\}$, then $\mu_{\infty,\phi}=\mu_{\Or(0)}$ and $h_{\mu_{\infty,\phi}}(f)=0$;
\item[\rm{(2)}] If $\phi_{1,1}>\max\{\phi_{0,0},\frac{1}{2}\left(\phi_{0,1}+\phi_{1,0}\right)\}$, then $\mu_{\infty,\phi}=\mu_{\Or(1)}$ and $h_{\mu_{\infty,\phi}}(f)=0$;
\item[\rm{(3)}] If $\frac{1}{2}\left(\phi_{0,1}+\phi_{1,0}\right)>\max\{\phi_{0,0},\phi_{1,1}\}$, then $\mu_{\infty,\phi}=\mu_{\Or(01)}$ and $h_{\mu_{\infty,\phi}}(f)=0$;

\item[\rm{(4)}] If $\phi_{0,0}=\phi_{1,1}=\frac{1}{2}\left(\phi_{0,1}+\phi_{1,0}\right)$, then $\phi$ is cohomologous to zero and $\mu_{\infty,\phi}=\mu_{\rm mme}$ is the unique measure of maximal entropy of $f$, i.e. the Bernoulli measure with weight $1/2$. In particular, $h_{\mu_{\infty,\phi}}(f)=\log 2$;
\item[\rm{(5)}] If $\phi_{0,0}=\frac{1}{2}\left(\phi_{0,1}+\phi_{1,0}\right)>\phi_{1,1}$, then $\mu_{\infty,\phi}$ is the unique measure of maximal entropy of the subshift of finite type with transition matrix $A=
\begin{pmatrix}
1&1\\
1&0
\end{pmatrix}$ and $h_{\mu_{\infty,\phi}}(f)=\log \tau$, where $\tau=\frac{1+\sqrt{5}}{2}$ is the Golden Mean;
\item[\rm{(6)}] If $\phi_{1,1}=\frac{1}{2}\left(\phi_{0,1}+\phi_{1,0}\right)>\phi_{0,0}$, then $\mu_{\infty,\phi}$ is the unique measure of maximal entropy on the subshift of finite type with transition matrix $A=
\begin{pmatrix}
0&1\\
1&1
\end{pmatrix}$ and $h_{\mu_{\infty,\phi}}(f)=\log \tau$;
\item[\rm{(7)}] If $\phi_{0,0}=\phi_{1,1}>\frac{1}{2}\left(\phi_{0,1}+\phi_{1,0}\right)$, then $\mu_{\infty,\phi}=\frac{1}{2}\left(\mu_{\Or(0)}+\mu_{\Or(1)}\right)$ and $h_{\mu_{\infty,\phi}}(f)=0$;

\end{enumerate}
\end{example}
We note that Example 1 was previously treated with different methods in \cite{Br}.
The cases {\rm (5)} and {\rm(6)} correspond to the so-called Golden Mean shift. By \cite{PP}, the measure with the maximal entropy is the Markov measure with probability vector $\begin{pmatrix}
\tau^2/(1+\tau^2)\\
1/(1+\tau^2)
\end{pmatrix}$ and transition matrix $\begin{pmatrix}
1/\tau&1/\tau^2\\
1&0
\end{pmatrix}$  and the analogous measure in (6). We refer to \cite{Br} and \cite[Proposition 17.14]{DGS}  for further details.

\subsection{The 3-symbol full-shift}

Let $X=\{0,1,2\}^{\mathbb{N}}$ and let $f:X\to X$ be the shift map. Let $\phi\in LC_2(X,\bR)$. In Table 1 below we list the $2$-elementary periodic orbits of $f$ grouped by periods. Together $f$ has $148$ periodic orbits that are $2$-elementary. 
Based on the data in Table 1 we  are able to characterize several of the zero-temperature measures as well as the associated components (i.e. the convex cones) of potentials (see Theorem A).

\begin{example}{\rm (Unique maximizing elementary periodic points as zero-temp-erature measures.)}
\label{example2a}Let $X=\{0,1,2\}^{\mathbb{N}}$ and let $f:X\to X$ be the shift map. Let $\phi\in LC_2(X,\bR)$. Then there are  only $8$ elementary periodic point measures 
 that  occur as zero-temperature measures in the classification in Theorem A (a), i.e.  $\ell_1=8$. Namely, we  have the following possibilities:
\begin{enumerate}
\item[\rm{(1)}] If $\phi_{0,0}>\mu_x(\phi)$ for all $x\in \EPer^2(f)\setminus\{\Or(0)\}$, then $\mu_{\infty,\phi}=\mu_{\Or(0)};$
\item[\rm{(2)}] If $\phi_{1,1}>\mu_x(\phi)$ for all $x\in \EPer^2(f)\setminus\{\Or(1)\}$, then $\mu_{\infty,\phi}=\mu_{\Or(1)};$
\item[\rm{(3)}] If $\phi_{2,2}>\mu_x(\phi)$ for all $x\in \EPer^2(f)\setminus\{\Or(2)\}$, then $\mu_{\infty,\phi}=\mu_{\Or(2)};$
\item[\rm{(4)}] $\frac{1}{2}\left(\phi_{0,1}+\phi_{1,0}\right)>\mu_x(\phi)$ for all $x\in \EPer^2(f)\setminus\{\Or(01),\Or(10)\}$, then $\mu_{\infty,\phi}=\mu_{\Or(01)};$
\item[\rm{(5)}] If $\frac{1}{2}\left(\phi_{1,2}+\phi_{2,1}\right)>\mu_x(\phi)$ for all $x\in \EPer^2(f)\setminus\{\Or(12),\Or(21)\}$, then $\mu_{\infty,\phi}=\mu_{\Or(12)};$
\item[\rm{(6)}] If $\frac{1}{2}\left(\phi_{0,2}+\phi_{2,0}\right)>\mu_x(\phi)$ for all $x\in \EPer^2(f)\setminus\{\Or(02),\Or(20)\}$, then $\mu_{\infty,\phi}=\mu_{\Or(02)};$
\item[\rm{(7)}] If $\frac{1}{3}(\phi_{1,0}+\phi_{02}+\phi_{2,1})>\mu_x(\phi)$ for all $x\in \EPer^2(f)\setminus\{\Or(102),\Or(021),\Or(210)\}$, then $\mu_{\infty,\phi}=\mu_{\Or(102)};$
\item[\rm{(8)}] If $\frac{1}{3}(\phi_{1,2}+\phi_{2,0}+\phi_{0,1})>\mu_x(\phi)$ for all $x\in \EPer^2(f)\setminus\{\Or(120),\Or(201),\Or(012)\}$, then $\mu_{\infty,\phi}=\mu_{\Or(120)};$
\end{enumerate}
\end{example}

{\rm
\begin{center}
\begin{tabular}{ |c|c| } 
 \hline
 \text{Period} &  \text{Generating segments of 2-elementary periodic orbits } \\ 
 \hline
\multirow{1}{1em}{1}
 & 0;1;2 \\ 
 \hline
\multirow{1}{1em}{2}
 & 10;12; 20  \\ 
 \hline
\multirow{1}{1em}{3}
& 100;101; 102; 112; 120; 122; 200; 202 \\
 \hline
 \multirow{2}{1em}{4} &1001;
  1002;
  1012;
  1020;
  1021;
  1022;
  1120;
  1122;\\
 & 1200;
  1202;
  1220;
  2002 \\
  \hline
  \multirow{3}{1em}{5}&
   10012;10020;10021;10022;10112;10121;10122;10200;\\
   &10201;10220;10221;11200;11202;11220;12002;12022;\\
   &12200;12202\\
   \hline
   \multirow{4}{1em}{6}&
100112;
  100121;
100201;
102001;
100220;
  102200;\\
  &
 101122;
    101221;
  101202;102012;
   120022;
  122002;\\
  &
  112022;
  112202;
  102201;
    100122;
   100221;
    102120;\\
    &
  112002;
    112200\\
     \hline
   \multirow{4}{1em}{7}&
 1001122;1001221;
  1002201;
  1022001;
   1120022;
  1122002;\\
 & 1001202;
  1002012;
  1002120;
  1012002;
  1020012;
  1021200;\\
  &
  1011202;
    1020121;
  1021120;
  1021201;
    1020112;
    1012021;\\
    &
 1012202;
  1021220;
  1020122;
  1012022;
  1022012;
  1022120\\
   \hline
   \multirow{6}{1em}{8}&
   10011202;
  10012021;
   10020112;
   10020121;
    10021120;
   10021201;\\
   &
  10112002;
    10120021;
    10200112;
      10200121;
        10211200;
          10212001;\\&
  10012022;
  10012202;
  10020122;
  10021220;
  10022012;
  10022120;\\
  &
  10120022;
  10122002;
  10200122;
  10212200;
  10220012;
 10221200;\\&
  10201122;
   10211220;
     10112202;
     10120221;
    10122021;
  10201221;\\
  &
   10220112;
   10220121;
  10221120;
  10221201;
  10112022; 
    10212201
  \\
 \hline
 \multirow{5}{1em}{9}&
100112022;
  100112202;
  100120221;
  100122021;
  100201122;\\
  &
  100201221;
  100211220;
  100212201;
  100220112;
  100220121;\\
  &
  100221120;
  100221201;
  101120022;
  101122002;
  101200221;\\
  &
  101220021;
  102001122;
  102001221;
  102112200;
  102122001;\\
  &
  102200112;
  102200121;
  102211200;
  102212001\\
  \hline
\end{tabular}
\\[0.5cm]
 {\bf Table 1.} The 2-elementary periodic orbits of $f$ (one generating segment per orbit).
\end{center}
}
\noindent

To see that the cases (1)-(8) in Example \ref{example2a} cover all possibilities in Theorem A (a) we notice that for any other $2$-elementary periodic point $x$ and any $\phi\in LC_2(X,\bR)$ one can 
represent $\mu_x(\phi)$ as a weighted convex combination of  the $\mu_{x_1}(\phi),\dots,\mu_{x_8}(\phi)$. For example, if $x=\Or(12202)$, then 
\begin{eqnarray*}
\mu_x(\phi)&=&\frac{1}{5}\left(\phi_{1,2}+\phi_{2,2}+\phi_{2,0}+\phi_{0,2}+\phi_{2,1}\right)\\
&=&\frac{1}{5}\left(\mu_{\Or(2)}(\phi)+2\mu_{\Or(12)}(\phi)+2\mu_{\Or(02)}(\phi)\right).
\end{eqnarray*}
This shows that if $\mu_x$  is a $\phi$-maximizing measure then so are $\mu_{\Or(2)},\mu_{\Or(12)}$ and $ \mu_{\Or(02)}$. By going through the list of points  in Table 1 one  can exclude the remaining $2$-elementary 
periodic point measures. It follows from Theorem A that the sets of functions defined in each of the cases (1)-(8) are open convex cones  whose union is  dense in $LC_2(X,\bR)$.

\begin{example}{\rm (Maximal entropy measures of non-discrete subshifts of finite type as zero-temperature measures.)}\label{example2b}
Consider the  subshift  $\sigma:X_A\rightarrow X_A$ with transition matrix 
\[A=\begin{pmatrix}
1&0&1    \\
0&1&1\\
1&1&1
\end{pmatrix}\]
and a function $\phi$ which is constant on $2$-cylinders. Since $\sigma\vert_{X_A}$ is transitive, it has a unique measure  $\mu=\mu_A$ of maximal entropy. Moreover, $h_\mu(f)=\log\left(1+\sqrt{2}\right)$. By \cite{PP}, $\mu$ is a Markov measure with probability vector $p_A$ and transition matrix $P_A$ given by  $$p_A=\begin{pmatrix}
\frac{1}{4}\\
\frac{1}{4}\\
\frac{1}{2}
\end{pmatrix}\, \text{ and }\, \, P_A=
 \begin{pmatrix}
\sqrt{2}-1&0&1-\frac{\sqrt{2}}{2}\\
0&\sqrt{2}-1&1-\frac{\sqrt{2}}{2}\\
2-\sqrt{2}&2-\sqrt{2}&\sqrt{2}-1
\end{pmatrix}.$$  
 If
\begin{eqnarray*}
&\phi_{0,0}=\phi_{1,1}=\phi_{2,2}=\frac{1}{2}\left(\phi_{0,2}+\phi_{2,0}\right)=\frac{1}{2}\left(\phi_{1,2}+\phi_{2,1}\right)\\
>&\max\left\{ \frac{1}{2}\left(\phi_{0,1}+\phi_{1,0}\right),\frac{1}{3}(\phi_{1,0}+\phi_{02}+\phi_{2,1}), \frac{1}{3}(\phi_{1,2}+\phi_{2,0}+\phi_{0,1})   \right\},
\end{eqnarray*}
then the zero-temperature measure $\mu_{\infty,\phi}=\mu$  . Similar results can be stated for the other $2$-cylinder transitive subshifts of finite type. 

\end{example}
 
Example \ref{example2b} illustrates an application of  Theorem A (b) to derive zero-temperature measures that maximize entropy on certain non-discrete transitive subshifts of finite type. 


\begin{example}{\rm (Multiple ergodic component zero-temperature measures.)}\label{example2c}\\
Let $X=\{0,1,2\}^{\mathbb{N}}$ and let $f:X\to X$ be the shift map. Let $\phi\in LC_2(X,\bR)$. If $\phi$ satisfies
\begin{equation}\label{eqex2}
\phi_{0,0}=\phi_{1,1}=\phi_{2,2}>\mu_x(\phi)
\end{equation}
for all $x\in \EPer^2(f)\setminus\{\Or(0),\Or(1),\Or(2)\} $, then by Theorem A (c),
\begin{equation}\label{eqex3}
\mu_{\infty,\phi}=a_{1,\phi}\mu_{\Or(0)}+a_{2,\phi}\mu_{\Or(1)}+a_{3,\phi}\mu_{\Or(2)},\end{equation}
where $0\leq a_{1,\phi},a_{2,\phi},a_{3,\phi}\leq 1$ and $a_{1,\phi}+a_{2,\phi}+a_{3,\phi}=1$. Moreover, there are only finitely many choices for the coefficients $a_{j,\phi}$. In the following we consider 3 particular cases and apply a direct computation method, see \cite{Je3}:



Case 1: We consider the potential $\phi_1$  given by
\[M_{\phi_1}=\begin{pmatrix}
4&0&0    \\
1&4&0\\
1&0&4
\end{pmatrix}.\]
By \cite{PP}, $\mu_{t\phi_1}$ is a Markov measure defined by the  probability vector $$p_t=\left(\frac{-1+\sqrt{1+8e^t}}{2\sqrt{1+8e^t}},\frac{1+\sqrt{1+8e^t}}{4\sqrt{1+8e^t}}, \frac{1+\sqrt{1+8e^t}}{4\sqrt{1+8e^t}}\right)^{\top}$$
 and stochastic matrix  
\[P_{t}=\frac{1}{\beta(e^t)}\begin{pmatrix}
2e^{4t}& -1+\sqrt{1+8e^t}&-1+\sqrt{1+8e^t}\\
\frac{4e^t}{-1+\sqrt{1+8e^t}}&2e^{4t}&2\\
\frac{4e^t}{-1+\sqrt{1+8e^t}}&2&2e^{4t}
\end{pmatrix},
\]
where $\beta(e^t)=1+2e^{4t}+\sqrt{1+8e^t}$. Taking the limit $t\rightarrow\infty$, we obtain
\[P_t\rightarrow \begin{pmatrix}1&0&0\\
0&1&0\\
0&0&1
\end{pmatrix}, \text{ and } p_t\rightarrow \begin{pmatrix}\frac{1}{2}\\\frac{1}{4}\\\frac{1}{4}
\end{pmatrix}.\]
It follows that $\mu_{\infty\cdot \phi_1}=\frac{1}{2}\mu_{\Or(0)}+\frac{1}{4}\mu_{\Or(1)}+\frac{1}{4}\mu_{\Or(2)}$.

Case 2: We consider the potential $\phi_2$  given by
\[M_{\phi_2}=\begin{pmatrix}
4&0&0    \\
0&4&0\\
0&0&4
\end{pmatrix}.\]
By \cite{PP} and a similar argument as in Case 1, we obtain  $$\mu_{\infty,\phi_2}=\frac{1}{3}\left(\mu_{\Or(0)}+\mu_{\Or(1)}+\mu_{\Or(2)}\right).$$

Case 3: Finally, we consider the potential $\phi_3$  given by
\[M_{\phi_3}=\begin{pmatrix}
4&0&0    \\
1&4&0\\
0&0&4
\end{pmatrix}.\]
By \cite{PP}, (see also \cite{Je3} for a similar example), one  can show that  $\mu_{\infty,\phi_3}=\frac{1}{2}\left(\mu_{\Or(0)}+\mu_{\Or(1)}\right)$.
In particular, precisely two of the coefficients $a_{j,\phi_3
}$ in \eqref{eqex3} are non-zero.

\end{example}

Cases (1)-(3) in Example \ref{example2c} illustrate certain possibilities for zero-temperature measures that are convex combinations of periodic  point measures. In particular, we show the possibility of non-symmetric coefficients as well as of one coefficient being zero.  We refer to \cite{Je} for further examples and to \cite{CGU} for the case of irrational coefficients for the $5$-symbol full shift. Besides these discrete applications of Theorem A (c), there are non-discrete applications in the literature with one of the coefficients $a_{j,\phi}$ being zero   \cite{Je3}.


\section{Examples related to Theorem B}
In this appendix we construct examples that illustrate Theorem B. We first establish in  Example \ref{example3}   the possibility
of a line segment face $F$ with a non-analytic entropy function $\cH\vert_{{\rm ri} F}$. Finally, we describe how to refine  Example \ref{example3} to establish that  $\cH\vert_{{\rm ri} F}$ is in general not even a $C^1$ function, see Example \ref{example44}.

\begin{example}\label{example3}
Let $\cA=\{0,\dots,5\}$ and $f=f_A:X_A\to X_A$ be the one-sided subshift of finite type  with alphabet $\cA$ and transition matrix $A$, where $A$ is defined by
\[
A=
\begin{pmatrix}
1&1&1&0&0&0\\
1&1&1&0&0&0\\
1&1&1&1&1&1\\
0&0&0&1&1&1\\
0&0&0&1&1&1\\
1&1&1&1&1&1
\end{pmatrix}.
\]
It follows from the definition of $A$ that $f$ is transitive.  Let $w_1,w_2,w_3\subset \bR^2$ defined by $w_1=(0,0), w_2=(1,0)$ and $w_3=(1/2,1)$ and let $\Delta$ denote the polyhedron (i.e. triangle) with vertices $w_1,w_2,w_3$. We define the function $\Phi:X_A\rightarrow\mathbb{R}^2$ by
\begin{equation}\label{defpotphi}
\Phi(\xi)=\begin{cases}
w_1& {\rm if}\,\,   \xi\in \cC_2(00)\cup \cC_2(01)\cup C_2(10)\cup C_2(44)\\
 w_2& {\rm if}\,\,   \xi\in \cC_2(33)\cup \cC_2(34)\cup C_2(43)\cup C_2(11)\\                      
                       w_3 & \,\,\,\,\,  \text{otherwise}.
            \end{cases}
\end{equation}
\end{example}
It follows from the construction that $\Phi$ is constant on cylinders of length $2$. Let $F$ be the line segment with endpoints $w_1$ and $w_2$.
\begin{theorem}
Let $f$ and $\Phi$ be as in Example \ref{example3}. Then $\R(\Phi)=\Delta$ and $\cH$ is not analytic on the interior of the face $F$.
\end{theorem}
\begin{proof}
Since $\conv(\Phi(X_A)) =\Delta$ and since for each $i=1,2,3$ the set $\Phi^{-1}(\{w_i\})$ contains a fixed point, the statement $\R(\Phi)=\Delta$ follows from the convexity of $\Delta$. Moreover, $F$ is a face of $\R(\Phi)$. Suppose $x=(x_n)$ is a periodic point with $\mu_x\in F$. It follows from the definition of $\Phi$ that $x_n\in \{0,1,3,4\}$
for all $n\geq 1$. 
Define
\[
A_1=
\begin{pmatrix}
1&1&0&0&0&0\\
1&1&0&0&0&0\\
0&0&0&0&0&0\\
0&0&0&0&0&0\\
0&0&0&0&0&0\\
0&0&0&0&0&0
\end{pmatrix}
\, \, \, {\rm and}\,\,\,
A_2=
\begin{pmatrix}
0&0&0&0&0&0\\
0&0&0&0&0&0\\
0&0&0&0&0&0\\
0&0&0&1&1&0&\\
0&0&0&1&1&0&\\
0&0&0&0&0&0
\end{pmatrix}
\]
It follows that $X_F=X_{A_1}\cup X_{A_2}$. Moreover, each $f|_{X_{A_i}}$ is a transitive subshift that is conjugate to the full-shift in 2 symbols. This shows that $\htop(f|_{X_{A_i}})=\log 2.$ Hence, $\htop(f|_{X_F})=\log 2$.
Let $\mu_i$ denote the unique measure of maximal entropy for $f|_{X_{A_i}}$, that is, the unique invariant measure $\mu_i$ on $X_{A_i}$ satisfying $h_{\mu_i}(f)=\log 2$. It follows that $\mu_i$ is the Bernoulli measure that assigns each cylinder of length 2 the measure $1/4$. We obtain by computation  that $v_1\eqdef\rv(\mu_1)=(1/4,0)$ and $v_2\eqdef\rv(\mu_2)=(3/4,0)$. Since $\cH$ is concave and $\cH|_F\leq \htop(f|_{X_F})=\log 2$ we may conclude that 
$\cH|_{[v_1,v_2]}\equiv \log 2$. Here $[v_1,v_2]$ denotes the closed line segment with endpoints $v_1$ and $v_2$. It is a consequence of the identity theorem for analytic functions that if $\cH$ were analytic on $\inn F$ then $\cH$ must be constant $\log 2$ on $F$.
We claim that $\cH(w)<\log 2$ for all $w\in F\setminus [v_1,v_2]$. Let $w\in F\setminus [v_1,v_2]$.
By symmetry it is sufficient to consider the case $w\in [0,1/4)$. Since the entropy map $\nu\mapsto h_{\nu}(f)$ is upper semi-continuous there exists $\nu\in \cM_{f|_{X_F}}$ with $\rv(\nu)=w$ and $h_\nu(f)=\cH(w)$.
Using that $X_F=X_{A_1}\cup X_{A_2}$ we can write $\nu=\lambda_1 \nu_1+\lambda_2\nu_2$ for some  $\nu_i\in \cM_{f|_{X_{A_i}}}$ and $\lambda_i\in [0,1]$ with $\lambda_1+\lambda_2=1$. If both, $\nu_1=\mu_1$ and $\nu_2=\mu_2$ then $\rv(\nu)\in [v_1,v_2]$ which is a contradiction. Thus, there must exist $i\in\{1,2\}$ with $\lambda_i>0$ and $\nu_i\not=\mu_i$; in particular $h_{\nu_i}(f)<\log 2$. Using that the entropy is affine, we conclude $\cH(w)=h_{\nu}(f)<\log 2$ which proves the claim. We obtain that $\cH$ is not analytic on $\ri F$. 
\end{proof}

\begin{remark}\label{remarkwer}
{\rm  (a)} In Example \ref{example3} the localized entropy function $\cH|_F$ is not analytic on the relative interior of the face $F$.  However, it is fairly straight-forward to show that  $\cH|_F$ is  a $C^1$-map. In particular, this statement follows from Theorem \ref{thmpwanalytic}.\\
{\rm (b)}  Let $\Phi$  be as in Example  \ref{example3}  and let  $\alpha=(0,-1)$. It is a consequence of Theorem \ref{KW} that  the zero-temperature measure $\mu_{\infty,\alpha\cdot\Phi}$ of the function $\alpha\cdot \Phi$ is a convex combination of the the measures $\mu_1$ and $\mu_2$. Using a symmetry argument similar to that  in 
\cite{KW4}, one can show that $\mu_{\infty,\alpha\cdot \Phi}=\frac{1}{2}(\mu_1+\mu_2)$. In particular, $\mu_{\infty,\alpha\cdot \Phi}$ is a non-ergodic zero-temperature measure
with positive entropy.
\end{remark}


The next example establishes the possibility of a line segment face $F$ with a relative interior point at which $\cH$ is not differentiable. 
\begin{example}\label{example44}Let $f:X\to X$ be a transitive subshift of finite type. Moreover, let $\Phi\in LC_k(X,\mathbb{R}^2)$ such that $\R(\Phi)$ has a
line segment face $F=[w_1,w_3]$ such that $X_F$ decomposes into $X_F=X_1\cup X_2 \cup X_3$ (see Equations \eqref{perface} and \eqref{sumtrans}) with the following properties: {\rm (a)} $\R(\Phi|_{X_i})=\{w_i\}$ for $i=1,3$ and $\R(\Phi|_{X_2})=\{w_2\}$ where $w_2\in(w_1,w_3)$; {\rm (b)} $\htop(f|_{X_2})>\max \{\htop(f|_{X_1}),\htop(f|_{X_3})\}$.\\
It now follows from Proposition \ref{concenv} that the graph of $\cH|_F$ is given by the line segments joining $(w_1, \htop(f|_{X_1}))$ with 
$(w_2, \htop(f|_{X_2}))$, and $(w_2, \htop(f|_{X_2}))$ with
$(w_3, \htop(f|_{X_3}))$ respectively. In particular, $\cH|_F$ is not differentiable at $w_2$.
\end{example}
We note that a function $\Phi$  satisfying the conditions in Example \ref{example44} can be obtained by slightly modifying the construction in
Example \ref{example3}. We leave the details to the reader.

\end{document}